\RequirePackage{fix-cm}
\documentclass[smallextended]{svjour3}       
\usepackage[utf8]{inputenc}

\usepackage{tabularx,ragged2e,booktabs,caption}
\newcolumntype{C}[1]{>{\Centering}m{#1}}

\smartqed  
\usepackage{xspace}
\usepackage{amsfonts}
\usepackage{amsmath}
\usepackage{enumerate}
\usepackage{framed}
\usepackage{graphicx}
\usepackage{lscape}
\usepackage{bm}
\usepackage{color}
\usepackage{multicol}
\usepackage{cite}
\usepackage[utf8]{inputenc}

\usepackage{tabularx,ragged2e,booktabs,caption}
\newcolumntype{C}[1]{>{\Centering}m{#1}}

\usepackage{amsmath}
\usepackage{amsfonts}
\usepackage{amssymb}
\usepackage{tikz}
\usepackage{subfig}
\usepackage{float}
\usepackage{graphicx}
\usepackage[english]{babel}
\usepackage{geometry}                
\usepackage[parfill]{parskip}    
\usepackage{epstopdf}
\usepackage{psfrag}

\renewcommand{\eqref}[1]{{(\ref{#1})}}

\newtheorem{rema}{Remark}
\usepackage{graphicx}
\begin{document}

\title{A novel high dimensional fitted scheme for stochastic optimal control problems}
%


\titlerunning{A novel high dimensional fitted scheme for stochastic optimal control problems}        

\author{ Christelle Dleuna Nyoumbi,
Antoine Tambue
}
\institute{Antoine Tambue (Corresponding author) \at
	Western Norway University of Applied Sciences,  Inndalsveien 28, 5063  Bergen, Norway,\\
	Tel.: +47 55 58 70 06, \email{antonio@aims.ac.za, antoine.tambue@hvl.no, tambuea@gmail.com}\\
   \and
    Christelle Dleuna Nyoumbi, \at
            Institut de Math\'{e}matiques et de Sciences Physiques de l'Universit\'{e} d'Abomey Calavi \\
              BP 613, Porto-Novo, B\'{e}nin
           \email{ christelle.dleuna@imsp-uac.org } 
  }
\date{Received: date / Accepted: date}

\maketitle

\begin{abstract}
Stochastic optimal principle leads to the resolution of a partial differential equation (PDE), namely the Hamilton-Jacobi-Bellman (HJB) equation. In general, this equation cannot be solved analytically, thus numerical 
algorithms are the only tools to provide accurate approximations. The aims of this paper is to introduce a novel fitted finite volume method  to solve  high  dimensional   degenerated HJB equation from stochastic optimal control problems in high dimension ($ n\geq 3$). 
The challenge  here is due to the nature of  our HJB equation which  is a degenerated second-order partial differential equation coupled with an optimization problem. 
\textcolor{blue}{ For such problems,  standard scheme such as finite difference method  losses its monotonicity and therefore the convergence toward the viscosity solution  may not be guarantee.}
 \textcolor{blue}{ We discretize the HJB equation  using the fitted finite volume  method, well known to tackle degenerated PDEs, while  the time discretisation is  performed  using the Implicit
	Euler  scheme.}.
We show that matrices resulting \textcolor{blue}{from spatial discretization   and  temporal  discretization  are  M--matrices}. Numerical results in finance demonstrating the accuracy  of the proposed numerical method comparing to the standard finite difference method are provided.
\end{abstract}

	
\section{Introduction}
The theory of optimal control for stochastic differential equations is mathematically challenging and
 it has been considered in many fields such as economics, engineering, biology and finance \cite{WH2,HP}. 
 Stochastic optimal control problems have been studied by many researches \cite {KL,JER,HPFH}. 
 In some cases, the well posedness of such problems have been studied using methods such as viscosity  and  minimax techniques (see \cite{Cra1,Cra2,Cra5}). In general, most of them do not have an
 explicit solution,  therefore there have been many attempts to develop novel methods for their approximations.
 Numerical approximation of stochastic optimal control problem is therefore an active research area and has attracted a lot of attentions \cite{Cra7,KNV1,KL1,KL,JER,HPFH}. 
 The keys challenge for solving HJB equation are the low regularity of the solution  and the lack of appropriate  numerical methods to tackle the degeneracy of the differential operator in HJB equation.
  Indeed adding to the standard issue that we usually have when solving degenerated PDE, we need to couple with an optimization problem at each point of the grid and for each time step.
 A standard approach is based on  Markov chain approximation, which  suffers from time step limitations due to stability issues \cite{Peter} as the method is indeed based on finite difference approach.
 Many stochastic optimal control problems such as Merton optimal problems  have degenerated linear operator when the spatial variables approach the region near to zero.
 This degeneracy has an adverse impact on the accuracy when the finite difference method is used to solve such optimal problems \cite{15, wilmott2005best} as the monotonicity of the scheme is usually lost. 
 However, when solving  HJB equation,  the monotonicity \textcolor{blue}{also} plays a key role to ensure the convergence of the numerical scheme toward the viscosity solution.
Indeed in high dimensional  Merton's control problem,  the matrix in the diffusion part is  lower rank  near the origin and it has been found in \cite{chistoph2019, chistoph2020} 
that the standard finite difference schemes become non monotone and may not converge to the viscosity solution of the HJB.
 To solve the degeneracy issue, a fitted finite volume have been proposed in  \cite{15} for one and two dimensional optimal control  problems.
 This method uses special technique called fitted technique to tackle the degeneracy. The scheme  have been initially  developed to solve Black Scholes PDEs for options  pricing (see \cite{WS}
 and references therein).
 In  \cite{15}, numerical experiments  have  been used  to demonstrate that the  fitted finite volume  scheme   is  more accurate  than the standard  finite difference approach  to approximate  one and two dimensional stochastic optimal problems. 
 To the best of our knowledge, even for Black Scholes PDEs for options  pricing,  fitted  technique for high dimensional domain ($n\geq 3$)  has be lacked in the literature.
 
 The aim of this research is to introduce the first fitted finite volume method  for stochastic optimal control problems in high dimensional domain ($n\geq 3$).
 This method is suitable to handle the degeneracy of the linear operator while solving numerically the HJB equation. 
 The method is coupled with implicit time-stepping method and the iterative method presented in \cite{HPFH} for optimization problem at every time step.  
 The merit of the method is that it is absolutely stable in time because of the implicit nature of the time discretisation and  yields a linear system with a positive-definite $M $-matrix,  this is  in contrast  of  the  standard  finite difference scheme.\\
 
 \textcolor{blue}{The novel contribution of our paper over the existing literature can be summarized as
 	\begin{itemize}
 		\item  We have upgraded  the fitted finite volume  technique  to  discretize  a more generalized HJB equation  coupled  with  the  implicit time-stepping method
 		for temporal discretization  method and the iterative method for  the optimization problem at every time step.  To best of our knowledge  such combination  has not yet proposed so far to solve  stochastic optimal control problems in high dimensional domain ($n\geq 3$).
 		\item We have  proved that the corresponding matrices after spatial  and temporal discretization  are  positive-definite $M $--matrices.
 		We have demonstrated  by numerical experiments that the proposed scheme   can be  more accurate than the standard  finite difference scheme.
 	\end{itemize}}
 
 The rest of the paper is organized as follows. The stochastic optimal control problems is introduced in  section\ref{sec1}. In section \ref{sec2}, we introduce the fitted finite volume in  high dimensional domain and show that the system matrix of the resulting discrete equations is an $M $-matrix. Section \ref{sec3} provides temporal discretization and optimization algorithm  for  spatial diiscretized HJB equation.
 In section \ref{sec4},  we present some numerical examples illustrating  the  accuracy of the proposed method comparing to the standard finite difference. Finally, in section \ref{sec6}, we summarise our finding.
\section{Preliminaries and formulation}
 \label{sec1}
 Let $\left( \Omega, \mathcal{F}, \mathbb{F}=(\mathcal{F}_t)_{t \geq 0}, \mathbb{P}\right)$  be a filtrated probability space. 
We consider the numerical approximation of the following  controlled Stochastic Differential Equation (SDE) defined  in $\mathbb{R}^{ n}$    by
 \begin{equation} \label{premier}
 \begin{split}
 & dx_s = b(s,x_s, \alpha_s) dt + \sigma (s, x_s, \alpha_s) d\omega_s,\,\,\,s\,\in (t, T]\\
 & x_t=x
 \end{split}
 \end{equation}
 where  
 \begin{equation}
 \begin{split}
 b&: [0, T] \times \mathbb{R}^n \times \mathcal{A} \rightarrow  \mathbb{R}^n \\ 
 & (t,x_t, \alpha_t) \rightarrow b(t,x_t,\alpha_t)
 \end{split}
 \end{equation}
 is the drift term  and
 \begin{equation}
 \begin{split}
 \sigma &:[0, T] \times \mathbb{R}^n\times \mathcal{A} \rightarrow  \mathbb{R}^{n\times d} \\
 & (t,x_t,\alpha_t)  \rightarrow \sigma(t,x_t,\alpha_t)
 \end{split}
 \end{equation} 
 the $ d $-dimensional diffusion coefficients. Note that $ \omega_t $ are $ d $-dimensional independent Brownian motion on $ \left(\Omega, \mathcal{F}, (\mathcal{F}_t)_{t \geq 0}, \mathbb{P} \right) $, the control $ \alpha =(\alpha_t)_{t \geq 0}$ is an $ \mathbb{F} $-adapted process, valued in  $ \mathcal{A} $ compact convex subset of $ \mathbb{R}^m \, (m \geq 1)$ and satisfying some integrability conditions and/or state constraints. Precise assumptions on $ b $ and $ \sigma $  to ensure the existence of the unique solution $
 x_t $ of (\ref{premier}) can be found in \cite{HP}.\\
 Given a function $ f $ from  $[0, T] \times \mathbb{R}^n\times \mathcal{A} $ into  $ \mathbb{R}$ and $ g $ from  $\mathbb{R}^n $ into  $ \mathbb{R}$, the performance functional  is defined as 
 \begin{equation}
 J(t, x, \alpha) =  \mathbb{E}\, \left[ \int_t^T f(s, x_s,\alpha) \,ds + g ( x_T) \right], \,\,\,\forall\,x\,\in\,\mathbb{R}^n.
 \end{equation} 
We assume that 
 \begin{equation}
 \mathbb{E}\, \left\lbrace\left[ \int_t^T f(s, x_s,\alpha) \,ds + g (x_T) \right]\right\rbrace < \infty.
 \end{equation} 
 The model problem consists to solve the following optimization   
 \begin{equation} \label{vienu}
 v(t, x) = \underset{\alpha \in \mathcal{A}}{\sup}\, J(t, x, \alpha), \,\,\,\,\forall\,x\,\in\,\mathbb{R}^n. 
 \end{equation} 
 By dynamic programming, the resulting Hamilton Jacobi-Bellamn (HJB) equation \cite{KL}) is given by
 \begin{equation}\label{merci}
 \begin{cases}
 \dfrac{ \partial v(t, x)}{ \partial  t} + \underset{\alpha \in \mathcal{A}}{\sup} \left[L^{\alpha} v(t, x) + f(t, x,\alpha)\right] = 0 \quad\text{on} \  [0,T)\times \mathbb{R}^n\\
 v(T, x) = g(x), \,\,\,\,x \,\in \mathbb{R}^n
 \end{cases}
 \end{equation}
 where 
 \begin{equation}\label{tous}
 L^{\alpha} v(t, x) = \sum_{i=1}^n (b(t, x,\alpha))_i \dfrac{ \partial v(t,x)}{ \partial  x_i}  + \sum_{i,j=1}^n ( a^{\alpha} (t,x))_{i,j}\,\dfrac{ \partial^2 v(t,x)}{\partial x_i\, \partial x_j},
 \end{equation}   
 and $ a^\alpha (t,x) = \dfrac{1}{2}\bigg(\sigma (t,x,\alpha)(\sigma(t,x,\alpha))^T\bigg)  $.
The resulting Hamilton-Jacobi-Bellman equation is typically a second order nonlinear partial differential equation, which can degenerate and therefore should  to solve accurately. 

  \section{Fitted finite volume method in three dimension HJB} 
 \label{sec2}
As we have already mentioned, even for Black Scholes PDEs for options pricing,  fitted  technique for three dimensional space  has be lacked in the literature to the best of our knowledge.
 The goal here is to update the technique in  \cite{WS}
 to three dimensional  HJB equation.
 
 Consider the more generalized HJB equation \eqref{merci} in dimension $ 3 $  which can be written in
 the form by setting $ \tau = T-t $
 \begin{equation} \label{vie1}
 \begin{split}
 &-\dfrac{ \partial v(\tau,x,y,z) }{\partial \tau} \\
 &+ \sup_{\alpha \in \mathcal{A}}\left[ \nabla\cdot \left( k (v(\tau,x,y,z))\right) + c(\tau,x,y,z,\alpha)\,v(\tau,x,y,z) \right] = 0,
 \end{split}
 \end{equation}
 where $k(v(\tau,x,y,z)) = A(\tau,x,y,z,\alpha)\cdot\nabla v(\tau,x,y,z)+ b(\tau,x,y,z,\alpha)\,v(\tau,x,y,z)$\,  with 
 \begin{equation}
 \begin{split}
 \label{A}
 b = (x\,b_1, y\,b_2, z\,b_3)^T,\quad
 A=\left[ \begin{array}{ccc}
 a_{11} & a_{12} & a_{13} \\
 a_{21} & a_{22}& a_{23}  \\
 a_{31} & a_{32}& a_{33}
 \end{array} \right].
 \end{split}
 \end{equation}
 \textcolor{blue}{Indeed this  divergence form is not a restriction as the differentiation is respect to $x, y$  and $z$
 and not respect to the control  $\alpha$, which may be discontinuous in some applications. We will assume that $a_{21}= a_{12}, a_{31}= a_{13}$ and $  a_{32}= a_{23} $.  We also define the following coefficients, which  will help us to build our scheme 
$ a_{11}(t,x,y,\alpha) = \overline{a}_1(t,x, y, \alpha)\,x^2, a_{22}(t,x,y,\alpha) = \overline{a}_2(t,x,y,\alpha) y^2,\; a_{33}(t,x,y,\alpha) = \overline{a}_3(t,x,y,\alpha) z^2$, $  a_{1 2}=a_{2 1} = d_1(t,x,y,\alpha)\, x\,y\,z$,\,$  a_{1 3}=a_{3 1} = d_2(t,x,y,\alpha)\, x\,y\,z $ and  $ a_{2 3}=a_{3 2} = d_3(t,x,y,\alpha)\, x\,y\,z$.
Although this initial value problem \eqref{vie1} is defined on
the unbounded region $ \mathbb{R}^3 $, for computational reasons we often restrict  to a bounded region.}
As usual the three dimensional domain is truncated  to $ I_{x}= [0,x_{\text{max}}] $,  $I_{y}= [0,y_{\text{max}}] $ and $ I_{z}= [0,z_{\text{max}}] $. The truncated domain will  be divided into $ N_1 $, $ N_2 $ and $ N_3 $ sub-intervals
 $$  I_{x_{i}} :=(x_i, x_{i+1}),\,\, I_{y_{j}} :=(y_j, y_{j+1}) ,\,\, I_{z_{k}} :=(z_k, z_{k+1}),$$
 $ i=0\cdots N_1-1,\,\,j=0\cdots N_2-1,\,\,\,k=0\cdots N_3-1$
 with $ 0 = x_{0} < x_{1}< \cdots \cdots< x_{N_1} = x_{\text{max}},\,$ \, $0 = y_{0} < y_{1}< \cdots \cdots< y_{N_2}= y_{\text{max}} $ and  $ 0 = z_{0} < z_{1}< \cdots \cdots< z_{N_3}= z_{\text{max}} $. This defines on  $ I_{x} \times I_{y} \times I_{z} $  a rectangular  mesh.
 By setting 
 \begin{equation}
 \begin{split}
 & x_{i+1/2} :=\dfrac{x_{i} + x_{i+1} }{2},\,\, x_{i-1/2} :=\dfrac{x_{i} + x_{i-1} }{2},\,\,y_{j+1/2} :=\dfrac{y_{j} + y_{j+1} }{2},\\   &\,y_{j-1/2} :=\dfrac{y_{j} + y_{j-1} }{2}, \,\, \,z_{k+1/2} :=\dfrac{z_{k} + z_{k+1} }{2},\,\,z_{k-1/2} :=\dfrac{z_{k} + z_{k-1} }{2},
 \end{split}
 \end{equation}
 for each $ i=1\cdots N_1-1$\,\, $ j=1\cdots N_2-1$ \,and each $ k=1\cdots N_3-1$. These mid-points form a second partition of $ I_{x} \times I_{y} \times I_{z}$ if we define $ x_{-1/2} = x_{0}$,\, $ x_{N_1+1/2} = x_{\text{max}}$,\,\, $ y_{-1/2} = y_{0}$,\, $ y_{N_2+1/2} = y_{\text{max}}$ and  $ z_{-1/2} = z_{0}$,\, $ z_{N_3+1/2} = z_{\text{max}}$. For each $i = 0, 1, \cdots ,N_1 $,\,\, $j = 0, 1, \cdots ,N_2 $ and $k = 0, 1,\cdots,N_3 $, we set  $h_{x_i} = x_{i+1/2} - x_{i-1/2} $, $h_{y_j} = y_{j+1/2} - y_{j-1/2} $, \, $h_{z_k} = z_{k+1/2} - z_{k-1/2} $ and define the grids points as  \[
 \mathcal{G}=\left\lbrace  (x_i,y_j,z_k):~1\leq i\leq N_1-1;\,\,\,1\leq j\leq N_2-1; \,\,\,1\leq k\leq N_3-1\right\rbrace.
 \]
 Integrating both size of (\ref{vie1}) over $ \mathcal{R}_{i,j,k}=\left[ x_{i-1/2}, x_{i+1/2} \right] \times \left[  y_{j-1/2}, y_{j+1/2}\right]\times \left[  z_{k-1/2}, z_{j+1/2}\right] $ we have
 \begin{equation}
 \begin{split}
 & - \int_{x_{i-1/2}}^{x_{i+1/2}} \int_{y_{j-1/2}}^{y_{j+1/2}} \int_{z_{k-1/2}}^{z_{k+1/2}} \dfrac{ \partial v }{ \partial \tau}\, dx\, dy\, dz \\ 
 & + \int_{x_{i-1/2}}^{x_{i+1/2}} \int_{y_{j-1/2}}^{y_{j+1/2}} \int_{z_{k-1/2}}^{z_{k+1/2}} \sup_{\alpha \in \mathcal{A}} \left[ \nabla \cdot \left( k(v)\right) + c\,v \right]\, dx\,dy\,dz = 0,
 \end{split}
 \end{equation}
 for $ i =1,2,\cdots N_1-1 $,\, $ j =1,2,\cdots N_2-1 $,\, $k =1,2,\cdots N_3-1 $.\\
 Applying the mid-points quadrature rule to the first and the last point terms, we obtain the above
 \begin{equation}\label{bjr}
 \begin{split}
& -\dfrac{\partial \, v_{i,j,k}(\tau) }{\partial\, \tau}\,l_{i,j,k}\\
 & + \sup_{\alpha \in \mathcal{A}} \left[ \int_{ \mathcal{R}_{i,j,k}}\nabla \cdot \left( k(v)\right)\,dx\,dy\,dz + c_{i,j,k}(\tau,\alpha)\, v_{i,j,k}(\tau)\,l_{i,j,k}\right] =0
 \end{split}
 \end{equation}
 for $ i= 1,2,\cdots N_1-1 $,\, $ j= 1,2,\cdots N_2-1 $,\,$ k= 1,2,\cdots N_3-1 $ where $ l_{i,j,k} = \left( x_{i+1/2} - x_{i-1/2} \right) \times \left(y_{j+1/2} - y_{j-1/2}\right) \times \left(   z_{k+1/2} - z_{k-1/2}\right) $ is the volume of $ \mathcal{R}_{i,j,k} $. Note that $ v_{i,j,k}(\tau) $ denotes the nodal approximation to $ v(\tau, x_{i}, y_{j}, z_k) $ at each point of the grid.
  
 We now consider the approximation of the middle term in \eqref{bjr}. Let $\bf n $ denote the unit vector outward-normal to $ \partial \mathcal{R}_{i,j,k} $. By Ostrogradski Theorem, integrating by parts and using the definition
 of flux $ k $, we have
 \begin{equation}\label{vol}
 \begin{split}
& \int_{\mathcal{R}_{i,j,k}} \nabla \cdot \left( k(v)\right) d x\,dy\,dz \\
& = \int_{\partial \mathcal{R}_{i,j,k}} k \cdot \text{\textbf{n}} \, ds \\ 
 &= \int_{\left(x_{i+1/2},y_{j-1/2},z_{k-1/2} \right)}^{\left(x_{i+1/2}, y_{j+1/2}, z_{k+1/2} \right)}\left(a_{11}\,\dfrac{\partial  v }{d x}+ a_{12}\,\dfrac{d v }{\partial y}+a_{13}\,\dfrac{\partial v }{\partial  z}+ x\,b_1\,v \right)d y\,d z  \\ 
 &- \int_{\left(x_{i-1/2},y_{j-1/2},z_{k-1/2} \right)}^{\left(x_{i-1/2}, y_{j+1/2}, z_{k+1/2} \right)} \left(a_{11}\,\dfrac{\partial  v }{\partial x}+ a_{12}\,\dfrac{\partial  v }{\partial y}+a_{13}\,\dfrac{\partial  v }{\partial z}+ x\,b_1\,v \right)\,d y\,dz \\ 
 & +\int_{\left(x_{i-1/2},y_{j+1/2},z_{k-1/2} \right)}^{\left(x_{i+1/2}, y_{j+1/2}, z_{k+1/2} \right)} \left(a_{21}\,\dfrac{\partial  v }{ \partial x}+ a_{22}\,\dfrac{\partial  v }{d y}+a_{23}\,\dfrac{\partial  v }{\partial  z}+ y\,b_2\,v \right)d x\,dz \\ 
 &- \int_{\left(x_{i-1/2},y_{j-1/2},z_{k-1/2} \right)}^{\left(x_{i+1/2}, y_{j-1/2}, z_{k+1/2} \right)} \left(a_{21}\,\dfrac{\partial v }{\partial  x}+ a_{22}\,\dfrac{\partial v }{ \partial y}+a_{23}\,\dfrac{\partial  v }{\partial  z}+ y\,b_2\,v \right)d x\,dz \\ 
 &+ \int_{\left(x_{i-1/2},y_{j-1/2},z_{k+1/2} \right)}^{\left(x_{i+1/2}, y_{j+1/2}, z_{k+1/2} \right)} \left(a_{31}\,\dfrac{\partial v }{\partial x}+ a_{32}\,\dfrac{\partial v }{ \partial y}+a_{33}\,\dfrac{\partial  v }{ \partial z}+ z\,b_3\,v \right)d x\,dy \\ 
 & - \int_{\left(x_{i-1/2},y_{j-1/2},z_{k-1/2} \right)}^{\left(x_{i+1/2}, y_{j+1/2}, z_{k-1/2} \right)} \left(a_{31}\,\dfrac{\partial v }{\partial x}+ a_{32}\,\dfrac{\partial v }{ \partial y}+a_{33}\,\dfrac{\partial v }{\partial  z}+z\, b_3\,v \right)d x\,dy.
 \end{split}
 \end{equation}
 Note  that
 \begin{eqnarray}
  \int_{\left(x_{1},y_{1},z_{1} \right)}^{\left(x_{1}, y_{2}, z_{2} \right)} f(x,y,z) dy dz:=\int_{y_1}^{y_2} \int_{z_1}^{z_2} f(x_1,y,z)  dz dy.
 \end{eqnarray}
 We shall look at (\ref{vol}) term by term. For the first term we want to approximate the integral by a constant, i.e,
 \begin{equation}
 \begin{split}
 &\int_{\left(x_{i+1/2},y_{j-1/2},z_{k-1/2} \right)}^{\left(x_{i+1/2}, y_{j+1/2}, z_{k+1/2} \right)}\left(a_{11}\,\dfrac{\partial  v }{\partial x}+ a_{12}\,\dfrac{\partial v }{ \partial y}+a_{13}\,\dfrac{\partial  v }{\partial z} + x\, b_1\,v \right)d y\,dz \\ 
 & \approx \left(a_{11}\,\dfrac{\partial v }{\partial  x}+ a_{12}\,\dfrac{\partial v}{\partial y}+ a_{13}\,\dfrac{ \partial v }{\partial  z} + x\,b_1\,v \right)\bigg|_{\left(x_{{i+1/2}},y_{j},z_k \right)} \cdot h_{y_{j}}\cdot h_{z_{k}}.
 \end{split}
 \end{equation}
 To achieve this, it is clear that we now need to derive approximations of the  $ k (v) \cdot \bf n $ defined above at the mid-point $ \left(x_{i+1/2},y_{j},z_{k} \right) $, of the interval $ I_{x_{i}} $ for $ i = 0, 1,\cdots N_1-1$.
  This discussion is divided into two cases for $ i \geq 1 $, and $ i = 0\, $ on the interval $ I_{x_0} = [0,x_{1}] $. This is really an extension of the two
 dimensional fitted finite volume presented  \cite{huangfitted2009}.
 
 \textbf{\underline{Case I}:} For $ i\geq 1 $.\\
 
 Let set $ a_{11}(\tau, x, y, z, \alpha) = \overline{a}_1(\tau, x, y, z, \alpha)\,x^2  $. We approximate the term $ \left(a_{11} \dfrac{\partial v}{\partial x}+ x\,b_1\,v \right) $ by solving the following two points boundary value problem
 \begin{equation}\label{vie2}
 \begin{split}
 & \left(\overline{a}_1(\tau, x_{i+1/2}, y_j, z_k, \alpha_{i,j,k})\,x \dfrac{ \partial v}{ \partial  x}+ {b_1}(\tau, x_{i+1/2}, y_j, z_k, \alpha_{i,j,k})\,v \right)'= 0, \\ 
 & v(x_{i},y_{j},z_k)= v_{i,j,k},\,\,\,\, v(x_{i+1}, y_{j}, z_k)= v_{i+1,j,k},
 \end{split}
 \end{equation}
 integrating (\ref{vie2}) yields the first-order linear equations
 \begin{equation}
 \overline{a}_1(\tau, x_{i+1/2}, y_j, z_k, \alpha_{i,j,k})\,x \dfrac{ \partial v}{ \partial x}+ {b_1}(\tau, x_{i+1/2}, y_j, z_k, \alpha_{i,j,k})\,v  = C_1,
 \end{equation}
 where $ C_1 $ denotes an additive constant. As in  (\cite{huangfitted2009}), we get 
 \begin{equation}
 C_1 = \dfrac{{b_1}_{i+1/2,j,k}(\tau,\alpha_{i,j,k})\,\left(x_{i+1}^{\beta_{i,j,k}(\tau)}\,v_{i+1,j,k}-x_{i}^{\beta_{i,j,k}(\tau)}\,v_{i,j,k} \right)}{x_{i+1}^{\beta_{i,j,k}(\tau)}-x_{i}^{\beta_{i,j,k}(\tau)}}.
 \end{equation}
 Therefore,
 \begin{equation}
 \begin{split}
 &  a_{11}\,\dfrac{\partial  v }{\partial x}+ a_{12}\,\dfrac{ \partial v }{ \partial y}+ a_{13}\,\dfrac{\partial v }{\partial z} +x\, b_1\,v \bigg|_{\left(x_{{i+1/2}},y_{j},z_k \right)}  \approx   \\ 
 &  x_{i+1/2}\,\left( \dfrac{{b_1}_{i+1/2,j,k}(\tau,\alpha_{i,j,k})\,\left(x_{i+1}^{\beta_{i,j,k}(\tau)}\,v_{i+1,j,k}-x_{i}^{\beta_{i,j,k}(\tau)}\,v_{i,j,k} \right)}{x_{i+1}^{\beta_{i,j,k}(\tau)}-x_{i}^{\beta_{i,j,k}(\tau)}} \right.\\ 
 &\left.+ {d_1}_{i,j,k}(\tau,\alpha_{i,j,k})\,y_j\,z_k\,\dfrac{ \partial v }{\partial y} \bigg|_{\left(x_{1_{i+1/2}},y_{j},z_k \right)} + {d_2}_{i,j,k}(\tau,\alpha_{i,j,k})\,y_j\,z_k\,\dfrac{\partial v }{\partial z} \bigg|_{\left(x_{1_{i+1/2}},y_{j},z_k \right)} \right),
 \end{split}
 \end{equation}
 where $ \beta_{i,j,k}(\tau) =\dfrac{{b_1}_{i+1/2,j,k}(\tau,\alpha_{i,j,k})}{{\overline{a}_1}_{i+1/2,j,k}(\tau,\alpha_{i,j,k})} $, $ a_{12} = a_{21} = d_{1}(\tau,x,y,z,\alpha)\,x\,y\,z  $ and $ a_{13} = a_{31} = d_{2}(\tau,x,y,z,\alpha)\,x\,y\,z $.\\
 Note that in this deduction, we have assumed that $ {b_1}_{i+1/2,j,k}(\tau,\alpha_{i,j,k}) \neq 0 $. Finally, we use the forward difference,
 \begin{equation*}
 \dfrac{\partial v }{\partial y}\bigg|_{\left(x_{1_{i+1/2}},y_{j},z_k \right)} \approx\dfrac{v_{i,j+1,k}-v_{i,j,k}}{h_{y_j}},\,\,\,\dfrac{\partial v }{\partial z}\bigg|_{\left(x_{1_{i+1/2}},y_{j},z_k \right)}  \approx\dfrac{v_{i,j,k+1}-v_{i,j,k}}{h_{z_k}}.
 \end{equation*}
 We finally have
 \small{
 	\begin{equation}\label{vol1}
 	\begin{split}
 	&\left[a_{11}\,\dfrac{\partial v }{ \partial x}+ a_{12}\,\dfrac{ \partial v}{ \partial y}+ a_{13}\,\dfrac{\partial  v}{ \partial z}+x\, b_1\,v \right]_{\left(x_{i+1/2},y_{j},z_k \right)}
 	\cdot h_{y_{j}}\cdot h_{z_{k}} \\ 
 	& \approx x_{i+1/2}\left( \dfrac{{b_1}_{i+1/2,j,k}(\tau,\alpha_{i,j,k})\,\left(x_{i+1}^{\beta_{i,j,k}(\tau)}\,v_{i+1,j,k}-x_{i}^{\beta_{i,j,k}(\tau)}\,v_{i,j,k} \right)}{x_{i+1}^{\beta_{i,j,k}(\tau)}- x_{i}^{\beta_{i,j,k}(\tau)}}+\right. \\ 
 	&\left. {d_1}_{i,j,k}(\tau,\alpha_{i,j,k})\,y_j\,z_k\,\dfrac{v_{i,j+1,k}-v_{i,j,k}}{h_{y_j}}+ {d_2}_{i,j,k}(\tau,\alpha_{i,j,k})\,y_j\,z_k\,\dfrac{v_{i,j,k+1}-v_{i,j,k}}{h_{z_k}} \right) \cdot {h_{y_j}}\cdot {h_{z_k}}.
 	\end{split}
 	\end{equation}}
 Similarly, the second term in (\ref{vol}) can be approximated by
 \begin{equation}\label{vol2}
 \begin{split}
 &\left[a_{11}\,\dfrac{\partial v }{\partial  x}+ a_{12}\,\dfrac{\partial v}{\partial y}+ a_{13}\,\dfrac{\partial v}{ \partial z}+ x\,b_1\,v \right]_{\left(x_{i-1/2},y_{j},z_k \right)}
 \cdot h_{y_{j}}\cdot h_{z_{k}} \\ 
 & \approx x_{i-1/2}\left( \dfrac{{b_1}_{i-1/2,j,k}(\tau,\alpha_{i,j,k})\,\left(x_{i}^{\beta_{i-1,j,k}(\tau)}\,v_{i,j,k}-x_{i-1}^{\beta_{i-1,j,k}(\tau)}\,v_{i-1,j,k} \right)}{x_{i}^{\beta_{i-1,j,k}(\tau)}- x_{i-1}^{\beta_{i-1,j,k}(\tau)}} +\right.\\ 
 &\left. {d_1}_{i,j,k}(\tau,\alpha_{i,j,k})\,y_j\,z_k\,\dfrac{v_{i,j+1,k}-v_{i,j,k}}{h_{y_j}} + {d_2}_{i,j,k}(\tau,\alpha_{i,j,k})\,y_j\,z_k\,\dfrac{v_{i,j,k+1}-v_{i,j,k}}{h_{z_k}} \right) \cdot {h_{y_j}}\cdot {h_{z_k}}.
 \end{split}
 \end{equation}
 \textbf{\underline{Case II:}} Approximation of the flux at $ i = 0 $ on the interval $ I_{x_0} = [0,x_1] $.
 Note that the analysis in  the case I does not apply to the approximation of the flux on $ I_{x_0} $ because it is the degenerated zone. Therefore, we reconsider the following form
 \begin{equation}\label{vie4}
 \begin{split}
 \bigg({\overline{a}_1}_{x_{1/2},j,k}(\tau,\alpha_{1,j,k})\,x \dfrac{\partial v}{\partial x}+ {b_1}_{x_{1/2},j,k}(\tau,\alpha)\,v \bigg)' \equiv C_2\,\,\,\textbf{in}\,\,[0,x_1] \\
 v(x_0, y_j, z_k)= v_{0,j,k},\,\,\,\, v(x_1,y_j,z_k) = v_{1,j,k},
 \end{split}
 \end{equation}
 where $ C_2 $ is an unknown constant to be determined. Integrating (\ref{vie4}), we find
 \begin{equation}
 \begin{split}
 &\left(\overline{a}_1(\tau,\alpha_{1,j,k})\,x \dfrac{\partial v}{\partial x}+ {b_1}\,v\right)\bigg|_{\left(x_{1_{1/2}},y_{j},z_k \right)} =\\ 
 &  \dfrac{1}{2}\left[ ({\overline{a}_1}_{x_{1/2},j,k}(\tau,\alpha_{1,j,k})+{b_1}_{x_{1/2},j,k}(\tau,\alpha_{1,j,k}))v_{1,j,k}-\right.\\ 
 &\left. ({\overline{a}_1}_{x_{1/2},j,k}(\tau,\alpha_{1,j,k})-{b_1}_{x_{1/2},j,k}(\tau,\alpha_{1,j,k}))v_{0,j,k}\right]
 \end{split}
 \end{equation}
 and deduce that
 \begin{equation}
 \begin{split}
 &\left[a_{11}\,\dfrac{\partial v }{\partial x}+ a_{12}\,\dfrac{\partial v}{\partial y}+ a_{13}\,\dfrac{\partial v}{\partial z} +x\, b_1\,v \right]_{\left(x_{1/2},y_{j},z_k \right)}
 \cdot h_{y_j}\cdot h_{z_k}\approx \\
 & x_{1/2}\left( \dfrac{1}{2}\left[ ({\overline{a}_1}_{x_{1/2},j,k} (\tau,\alpha_{1,j,k})+{b_1}_{x_{1/2},j,k}(\tau,\alpha_{1,j,k}))\,v_{1,j,k}\right. \right. \\  
 & \left.\left. -({\overline{a}_1}_{x_{1/2},j,k}( \tau,\alpha_{1,j,k})-{b_1}_{x_{1/2},j,k}(\tau,\alpha_{1,j,k}))\,v_{0,j,k}\right] \right.\\ 
 &\left.+ {d_1}_{1,j,k}(\tau,\alpha_{1,j,k})\,y_j\,z_k\,\dfrac{v_{1,j+1,k}-v_{1,j,k}}{h_{y_j}} +\right.\\ 
 &\left. {d_2}_{1,j,k}(\tau,\alpha_{1,j,k})\,y_j\,z_k\,\dfrac{v_{1,j,k+1}-v_{1,j,k}}{h_{z_k}} \right) \cdot h_{y_j}\cdot h_{z_k}.
 \end{split}
 \end{equation}
 \begin{rema}
 	Notice that if $ I_{x}= [\zeta,x_{\text{max}}]$ with  $\zeta 
 	\neq 0  $, we do not  need to truncate the interval $ I_x $, we just apply the fitted finite volume method directly   as for $ i \geq 1$.
 \end{rema}
 \textbf{\underline{Case III}:} For $ j\geq 1 $. 
 For the third term in (\ref{vol}) we want to approximate the integral by a constant, i.e,
 \begin{equation}
 \begin{split}
 &\int_{\left(x_{i-1/2},y_{j+1/2},z_{k-1/2} \right)}^{\left(x_{i+1/2}, y_{j+1/2}, z_{k+1/2} \right)}\left(a_{21}\,\dfrac{\partial v }{\partial x}+ a_{22}\,\dfrac{\partial v }{\partial y}+a_{23}\,\dfrac{\partial v }{\partial z}+ y\,b_2\,v \right)d x\,dz \\ 
 & \approx \left(a_{21}\,\dfrac{\partial v }{\partial x}+ a_{22}\,\dfrac{\partial v}{\partial y}+ a_{23}\,\dfrac{\partial v }{\partial z} + y\,b_2\,v \right)|_{\left(x_{i},y_{j+1/2},z_k \right)} \cdot h_{x_{i}}\cdot h_{z_{k}}.
 \end{split}
 \end{equation}
 Following the same procedure for the case I of this section, we find that
 \begin{equation}\label{vol1'}
 \begin{split}
 &\left[a_{21}\,\dfrac{\partial v }{\partial x}+ a_{22}\,\dfrac{\partial v}{\partial y}+ a_{23}\,\dfrac{\partial v}{ \partial z}+y\, b_2\,v \right]_{\left(x_{i},y_{j+1/2},z_k \right)}
 \cdot h_{x_{i}}\cdot h_{z_{k}} \\
 &\approx y_{j+1/2}\left( \dfrac{{b_2}_{i,j+1/2,k}(\tau,\alpha_{i,j,k})\,\left(y_{j+1}^{{\beta_1}_{i,j,k}(\tau)}\,v_{i,j+1,k}-y_{j}^{{\beta_1}_{i,j,k}(\tau)}\,v_{i,j,k} \right)}{y_{j+1}^{{\beta_1}_{i,j,k}(\tau)}- y_{j}^{{\beta_1}_{i,j,k}(\tau)}} + \right. \\
 & \left. {d_1}_{i,j,k}(\tau,\alpha_{i,j,k})\,x_i\,z_k\,\dfrac{v_{i+1,j,k}-v_{i,j,k}}{h_{x_i}} + {d_3}_{i,j,k}(\tau,\alpha_{i,j,k})\,x_i\,z_k\,\dfrac{v_{i,j,k+1}-v_{i,j,k}}{h_{z_k}} \right) \cdot {h_{x_i}}\cdot {h_{z_k}},
 \end{split}
 \end{equation}
 where $ {\beta_1}{i,j,k}(\tau) =\dfrac{{b_2}_{i,j+1/2,k}(\tau,\alpha_{i,j,k})}{{\overline{a}_2}_{i,j+1/2,k}(\tau,\alpha_{i,j,k})} $,\,\, $ a_{22}(\tau, x, y, z,\alpha) = {\overline{a}_2}(\tau,x,y ,z,\alpha)\,y^2,  $\\
 and $ a_{23} = a_{32} = d_{3}(\tau,x,y,z,\alpha)\,x\,y\,z$. Similary, the fourth term in (\ref{vol}) can be approximated by
 \begin{equation}\label{vol2'}
 \begin{split}
 &\left[a_{21}\,\dfrac{\partial v }{d\partial x}+ a_{22}\,\dfrac{\partial v}{\partial y}+ a_{23}\,\dfrac{\partial v}{\partial z}+y\, b_2\,v \right]_{\left(x_{i},y_{j-1/2},z_k \right)}
 \cdot h_{x_{i}}\cdot h_{z_{k}}  \\
 &\approx y_{j-1/2}\left( \dfrac{{b_2}_{i,j-1/2,k}(\tau,\alpha_{i,j,k})\,\left(y_{j}^{{\beta_1}_{i,j-1,k}(\tau)}\,v_{i,j,k}-y_{j-1}^{{\beta_1}_{i,j-1,k}(\tau)}\,v_{i,j-1,k} \right)}{y_{j}^{{\beta_1}_{i,j-1,k}(\tau)}- y_{j-1}^{{\beta_1}_{i,j-1,k}(\tau)}} +\right.\\ 
 & \left. {d_1}_{i,j,k}(\tau,\alpha_{i,j,k})\,x_i\,z_k\,\dfrac{v_{i+1,j,k}-v_{i,j,k}}{h_{x_i}}+ {d_3}_{i,j,k}(\tau,\alpha_{i,j,k})\,x_i\,z_k\,\dfrac{v_{i,j,k+1}-v_{i,j,k}}{h_{z_k}} \right) \cdot {h_{x_i}}\cdot {h_{z_k}}.
 \end{split}
 \end{equation}
 \textbf{\underline{Case IV:}} Approximation of the flux at $ I_{y_0} $ i.e for $j= 0 $. Using the same procedure  for the approximation of the flux at $ I_{x_0} $, we deduce that
 \begin{equation}
 \begin{split}
 &\left[a_{21}\,\dfrac{\partial v }{ \partial x}+ a_{22}\,\dfrac{ \partial v}{\partial y}+ a_{23}\,\dfrac{\partial v}{\partial z} + y\,b_2\,v \right]_{\left(x_{i},y_{1/2},z_k \right)}
 \cdot h_{x_i}\cdot h_{z_k} \approx \\
 & y_{1/2}\left( \dfrac{1}{2}\left[ ({\overline{a}_2}_{i,y_{1/2},k}(\tau,\alpha_{i,1,k})+{b_2}_{i,y_{1/2},k}(\tau,\alpha_{i,1,k}))v_{i,1,k}\right.\right. \\ 
 & \left.\left.-({\overline{a}_2}_{i,y_{1/2},k}(\tau,\alpha_{i,1,k})-{b_2}_{i,y_{1/2},k}(\tau,\alpha_{i,1,k}))\,v_{i,0,k}\right] \right.\\ 
 &\left.+ {d_1}_{i,1,k}(\tau,\alpha_{i,1,k})\,x_i\,z_k\,\dfrac{v_{i+1,1,k}-v_{i,1,k}}{h_{x_i}} +\right.\\ 
 &\left. {d_3}_{i,1,k}(\tau,\alpha_{i,1,k})\,x_i\,z_k\,\dfrac{v_{i,1,k+1}-v_{i,1,k}}{h_{z_k}} \right) \cdot h_{x_i}\cdot h_{z_k}.
 \end{split}
 \end{equation}
 For the fifth term in (\ref{vol}) we want to approximate the integral with  a constant.  Following the same procedure  as in the case I and  case III, we have
 \begin{equation}\label{vol11'}
 \begin{split}
 &\left[a_{31}\,\dfrac{\partial v }{\partial  x}+ a_{32}\,\dfrac{\partial  v}{\partial  y}+a_{33}\,\dfrac{\partial  v}{\partial  z}+z\, b_3\,v \right]_{\left(x_{i},y_{j},z_{k+1/2} \right)}
 \cdot h_{x_{i}}\cdot h_{y_{j}}  \\
 &\approx z_{k+1/2}\left( \dfrac{{b_3}_{i,j,k+1/2}(\tau,\alpha_{i,j,k})\,\left(z_{k+1}^{{\beta_2}_{i,j,k}(\tau)}\,v_{i,j,k+1}-z_{k}^{{\beta_2}_{i,j,k}(\tau)}\,v_{i,j,k} \right)}{z_{k+1}^{{\beta_2}_{i,j,k}(\tau)}- z_{k}^{{\beta_2}_{i,j,k}(\tau)}} +\right. \\ 
 & \left. {d_2}_{i,j,k}(\tau,\alpha_{i,j,k})\,x_i\,y_j\,\dfrac{v_{i+1,j,k}-v_{i,j,k}}{h_{x_i}}+ {d_3}_{i,j,k}(\tau,\alpha_{i,j,k})\,x_i\,y_j\,\dfrac{v_{i,j+1,k}-v_{i,j,k}}{h_{y_j}} \right) \cdot {h_{x_i}}\cdot {h_{y_j}},
 \end{split}
 \end{equation}
 where $ {\beta_2}{i,j,k}(\tau) =\dfrac{{b_3}_{i,j,k+1/2}(\tau,\alpha_{i,j,k})}{\bar{a_3}_{i,j,k+1/2}(\tau,\alpha_{i,j,k})} $,\,\, $ a_{33}(\tau, x, y, z,\alpha) = \overline{a}_3(\tau,x,y ,z,\alpha)\,z^2  $.\\ 
 Similarly, the sixth term in (\ref{vol}) can be approximated by
 \begin{equation}\label{vol22'}
 \begin{split}
 &\left[a_{31}\,\dfrac{\partial  v }{\partial x}+ a_{32}\,\dfrac{ \partial v}{ \partial y}+a_{33}\,\dfrac{\partial  v}{\partial z}+ z\,b_3\,v \right]_{\left(x_{i},y_{j},z_{k-1/2} \right)}
 \cdot h_{x_{i}}\cdot h_{y_{j}}  \\
 &\approx z_{k-1/2}\left( \dfrac{{b_3}_{i,j,k-1/2}(\tau,\alpha_{i,j,k})\,\left(z_{k}^{{\beta_2}_{i,j,k-1}(\tau)}\,v_{i,j,k}-z_{k-1}^{{{\beta_2}_{i,j,k-1}(\tau)}}\,v_{i,j,k-1} \right)}{z_{k}^{{\beta_2}_{i,j,k-1}(\tau)}- z_{k-1}^{{{\beta_2}_{i,j,k-1}(\tau)}}}+ \right. \\ 
 &\left. {d_2}_{i,j,k}(\tau,\alpha_{i,j,k})\,x_i\,y_j\,\dfrac{v_{i+1,j,k}-v_{i,j,k}}{h_{x_i}}+ {d_3}_{i,j,k}(\tau,\alpha_{i,j,k})\,x_i\,y_j\,\dfrac{v_{i,j+1,k}-v_{i,j,k}}{h_{y_j}} \right) \cdot {h_{x_i}}\cdot {h_{y_j}}. \end{split}
 \end{equation}
 \textbf{\underline{Case V:}} Approximation of the flux at $ I_{z_0} $. Using the same procedure for the Approximation of the flux at $ I_{z_0} $, we deduce that
 \begin{equation}
 \begin{split}
 &\left[a_{31}\,\dfrac{\partial  v }{\partial  x}+ a_{32}\,\dfrac{ \partial  v}{\partial  y}+ a_{33}\,\dfrac{\partial  v}{\partial  z} + z\,b_3\,v \right]_{\left(x_{i},y_{j},z_{1/2} \right)}
 \cdot h_{x_i}\cdot h_{y_j} \approx \\
 &  z_{1/2}\left( \dfrac{1}{2}\left[ (\bar{a_3}_{i,j,z_{1/2}}(\tau,\alpha_{i,j,1}) + {b_3}_{i,j,z_{1/2}}(\tau,\alpha_{i,j,1}))\,v_{i,j,1} \right. \right. \\ 
 & \left.\left.-(({\overline{a}_3}_{i,j,z_{1/2}}(\tau,\alpha_{i,j,1}) - {b_3}_{i,j,z_{1/2}}(\tau,\alpha_{i,j,1}))v_{i,j,0}\right] + \right.\\ \nonumber
 &\left. {d_2}_{i,j,1}(\tau,\alpha_{i,j,1})\,x_i\,y_j\,\dfrac{v_{i+1,j,1}-v_{i,j,1}}{h_{x_i}} + {d_3}_{i,j,1}(\tau,\alpha_{i,j,1}) \,x_i\,y_j\,\dfrac{v_{i,j+1,1}-v_{i,j,1}}{h_{y_j}} \right) \cdot h_{x_i}\cdot h_{y_j}.
 \end{split}
 \end{equation}
 Equation (\ref{bjr}) becomes by replacing the flux by his value for $ i = 1,\cdots,N_1-1 $, \,$ j = 1,\cdots,N_2-1 $,  \, $ k = 1,\cdots,N_3-1 $ and $ N =(N_1-1)\times (N_2-1)\times (N_3-1) $
 \begin{equation}\label{p5}
 \begin{cases}
 -\dfrac{d\, v_{i,j,k}(\tau)}{d\, \tau} 
 +  \underset{\alpha_{i,j,k} \in \mathcal{A}^{N}}{\sup}\, \left[ e_{i-1,j,k}^{i,j,k}\, v_{i-1,j,k} +e_{i,j,k}^{i,j,k}\, v_{i,j,k}+ e_{i+1,j,k}^{i,j,k}\, v_{i+1,j,k} \right. \\
  \left. + e_{i,j-1,k}^{i,j,k}\, v_{i,j-1,k}+e_{i,j+1,k}^{i,j,k} v_{i,j+1}+ e_{i,j,k-1}^{i,j,k}\, v_{i,j-1,k}+e_{i,j,k+1}^{i,j,k} v_{i,j,k+1} \right] = 0,\\
 ~~~\mbox{with}~~ \,\,\,\,  \textbf{v}(0) \,\,\,\text{given},
 \end{cases}
 \end{equation}
 This can be rewritten as the Ordinary Differential Equation (ODE) coupled with optimization
 {\small{
 		\begin{equation}\label{pan1}
 		\begin{cases}
 		\dfrac{d\, \textbf{v}(\tau)}{d\, \tau}  =\underset{\alpha \in \mathcal{A}^N}{\sup}\,\left[A (\tau,\alpha)\,\textbf{v}(\tau) + G(\tau, \alpha) \right] \\
 		~~~\mbox{with}~~ \,\,\,\,  \textbf{v}(0) \,\,\,\text{given},
 		\end{cases}
 		\end{equation}
 		or
 		\begin{equation}\label{p4}
 		\begin{cases}
 		\dfrac{d\,\textbf{v}(\tau)}{d\,\tau} + \underset{\alpha \in \mathcal{A}^N}{\inf}\,\left[E (\tau, \alpha)\,\textbf{v}(\tau) + F(\tau,\alpha) \right] = 0, \\
 		~~~\mbox{with}~~ \,\,\,\,  \textbf{v}(0)\,\,\,\text{given},
 		\end{cases}
 		\end{equation} 
 		where  $ A (\tau,\alpha) = - E (\tau,\alpha)$  is an $N\times N$ matrix, $  \mathcal{A}^{N} = \underset{(N_{1}-1)\times(N_2-1)\times(N_3-1)}{\underbrace{\mathcal{A} \times \mathcal{A}\times \cdots \times \mathcal{A}}} $ \,\, $G (\tau, \alpha) =-F(\tau, \alpha) $ depends of  the boundary condition  and the term $c$,  
 		$\textbf{v}= \left(v_{i,j,k}\right) $.
 		By setting
 		$n_1=N_1-1,\; n_2=N_2-1; \;n_3=N_3-1, \;\; I:=I (i,j,k)= i + (j-1)n_1 +(k-1)n_1 n_2$ and  $J:=J (i',j',k')=  i' + (j'-1)n_1 +(k'-1)n_1 n_2 $, we have 
 		$E (\tau, \alpha) (I,J)= \left(e_{i',j',k'}^{i,j,k}\right) $,  $i', i= 1,\cdots, N_1-1 $, \,\, $ j',j = 1,\cdots, N_2-1 $\,\, and\,\, $k', k = 1,\cdots, N_3-1 $  where the coefficients are defined by 
 		\begin{equation}
 		\begin{split}
 		& e_{i+1,j,k}^{i,j,k} = - \dfrac{{d_2}_{i,j,k}(\tau, \alpha_{i,j,k})\,x_i\,y_j}{h_{x_i}} - \dfrac{{d_1}_{i,j,k}(\tau, \alpha_{i,j,k})\,x_i\,z_k}{h_{x_i}}\\
 		& - x_{i+1/2}\dfrac{{b_1}_{i+1/2,j,k}(\tau, \alpha_{i,j,k})\,x_{i+1}^{\beta_{i,j,k}(\tau)} }{h_{x_i}\left(x_{i+1}^{\beta_{i,j,k}(\tau)}- x_{i}^{\beta_{i,j,k}(\tau)}\right)},\\
 		& e_{i-1,j, k}^{i,j,k} = - x_{i-1/2}\dfrac{{b_1}_{i-1/2,j,k}(\tau, \alpha_{i,j,k})\,x_{i-1}^{\beta_{i-1,j,k}(\tau)}}{h_{x_i}\left(x_{i}^{\beta_{i-1,j,k}(\tau)}- x_{i-1}^{\beta_{i-1,j,k}(\tau)}\right)},\\ 
 		& e_{i,j+1,k}^{i,j,k} = - \dfrac{{d_1}_{i,j,k}(\tau, \alpha_{i,j,k})\,y_j\,z_k}{h_{y_j}} - \dfrac{{d_3}_{i,j,k}(\tau, \alpha_{i,j,k})\,x_i\,y_j}{h_{y_j}} \\
 		&- y_{j+1/2}\dfrac{{b_2}_{i,j+1/2,k}(\tau, \alpha_{i,j,k})\,y_{j+1}^{{\beta_1}_{i,j,k}(\tau)} }{h_{y_j}\left(y_{j+1}^{{\beta_1}_{i,j,k}(\tau)}- y_{j}^{{\beta_1}_{i,j,k}(\tau)}\right)},\\
 		&  e_{i,j-1,k}^{i,j,k} = - y_{j-1/2}\dfrac{{b_2}_{i,j-1/2,k}(\tau, \alpha_{i,j,k})\,y_{j-1}^{{\beta_1}_{i,j-1,k}(\tau)}}{h_{y_j}\left(y_{j}^{{\beta_1}_{i,j-1,k}(\tau)}- y_{j-1}^{{\beta_1}_{i,j-1,k}(\tau)}\right)},\\
 		&  e_{i,j,k+1}^{i,j,k} = - \dfrac{{d_2}_{i,j,k}(\tau, \alpha_{i,j,k})\,y_j\,z_k}{h_{z_k}} - \dfrac{{d_3}_{i,j,k}(\tau, \alpha_{i,j,k})x_i\,z_k}{h_{z_k}} \\
 		&- z_{k+1/2}\dfrac{{b_3}_{i,j,k+1/2}(\tau, \alpha_{i,j,k})\,z_{k+1}^{{\beta_2}_{i,j,k}(\tau)} }{h_{z_k}\left(z_{k+1}^{{\beta_2}_{i,j,k}(\tau)}- z_{k}^{{\beta_2}_{i,j,k}(\tau)}\right)},\\
 		&  e_{i,j,k-1}^{i,j,k} = - z_{k-1/2}\dfrac{{b_3}_{i,j,k-1/2}(\tau, \alpha_{i,j,k})\,z_{k-1}^{{\beta_2}_{i,j,k-1}(\tau)} }{h_{z_k}\left(z_{k}^{{\beta_2}_{i,j,k-1}(\tau)}- z_{k-1}^{{\beta_2}_{i,j,k-1}(\tau)}\right)},
 		\end{split}
 		\end{equation}
 and 
 \begin{equation}
 \begin{split}
 & 	e_{i,j,k}^{i,j,k} \\
 & = 
 x_{i-1/2}\dfrac{{b_1}_{i-1/2,j,k}(\tau, \alpha_{i,j,k})\,x_{i}^{\beta_{i-1,j,k}(\tau)}}{h_{x_i}\left(x_{i}^{\beta_{i-1,j,k}(\tau)}- x_{i-1}^{\beta_{i-1,j,k}(\tau)}\right)} + y_{j-1/2}\dfrac{{b_2}_{i,j-1/2,k}(\tau, \alpha_{i,j,k})\,y_{j}^{{\beta_1}_{i,j-1,k}(\tau)}}{h_{y_j}\left(y_{j}^{{\beta_1}_{i,j-1,k}(\tau)}- y_{j-1}^{{\beta_1}_{i,j-1,k}(\tau)}\right)} \\ 
 & +\dfrac{{d_3}_{i,j,k}(\tau, \alpha_{i,j,k})\,x_i\,y_j}{h_{y_j}} +\dfrac{{d_3}_{i,j,k}(\tau, \alpha_{i,j,k})\,x_i\,z_k}{h_{z_k}} + \dfrac{{d_2}_{i,j,k}(\tau, \alpha_{i,j,k})\,x_i\,y_j}{h_{x_i}} \\ 
 & - c_{i,j,k}(\tau, \alpha_{i,j,k}) +\dfrac{{d_1}_{i,j,k}(\tau, \alpha_{i,j,k})\,y_j\,z_k}{h_{y_j}}+\dfrac{{d_2}_{i,j,k}(\tau,\alpha_{i,j,k})\,y_j\,z_k}{h_{z_k}}+\dfrac{{d_1}_{i,j,k}(\tau,\alpha_{i,j,k})\,x_i\,z_k}{h_{x_i}}\\
 &+ z_{k-1/2}\dfrac{{b_3}_{i,j,k-1/2}(\tau, \alpha_{i,j,k})\,z_{k}^{{\beta_2}_{i,j,k-1}(\tau)} }{h_{z_k}\left(z_{k}^{{\beta_2}_{i,j,k-1}(\tau)}- z_{k-1}^{{\beta_2}_{i,j,k-1}(\tau)}\right)}+z_{k+1/2}\dfrac{{b_3}_{i,j,k+1/2}(\tau, \alpha_{i,j,k})\,z_{k}^{{\beta_2}_{i,j,k}(\tau)} }{h_{z_k}\left(z_{k+1}^{{\beta_2}_{i,j,k}(\tau)}- z_{k}^{{\beta_2}_{i,j,k}(\tau)}\right)} \\ 
 &	+ x_{i+1/2}\dfrac{{b_1}_{i+1/2,j,k}(\tau, \alpha_{i,j,k})\,x_{i}^{\beta_{i,j,k}(\tau)} }{h_{x_i}\left(x_{i+1}^{\beta_{i,j,k}(\tau)}- x_{i}^{\beta_{i,j,k}(\tau)}\right)}+y_{j+1/2}\dfrac{{b_2}_{i,j+1/2,k}(\tau, \alpha_{i,j,k})\,y_{j}^{{\beta_1}_{i,j,k}(\tau)} }{h_{y_j}\left(y_{j+1}^{{\beta_1}_{i,j,k}(\tau)}- y_{j}^{{\beta_1}_{i,j,k}(\tau)}\right)}
 \end{split}
 \end{equation}
 for $ i = 2,\cdots, N_1-1 $, $ j = 2,\cdots, N_2-1 $  and $ k = 2,\cdots, N_3-1 $ and
 \begin{equation}
 \begin{split}
 e_{0,j,k}^{1,j,k}& = - \dfrac{1}{2\,x_2 } \,x_{1}(\bar{a_1}_{x_{1/2},i,j}(\tau,\alpha_{1,j,k})- {b_1}_{x_{1/2},j,k}(\tau,\alpha_{1,j,k}))\,v_{0,j,k},\\
 e_{1,j,k}^{1,j,k} & = \dfrac{1}{2\,x_2 } \,x_{1}( \bar{a_1}_{x_{1/2},i,j}(\tau,\alpha_{1,j,k})+ {b_1}_{x_{1/2},j,k}(\tau,\alpha_{1,j,k})) - \dfrac{1}{3}\,c_{1,j,k}(\tau,\alpha_{1,j,k})   \\ 
 &  +\dfrac{{d_2}_{1,j,k}(\tau,\alpha_{1,j,k})\,x_1\,y_j}{h_{x_1}} +  \dfrac{{d_1}_{1,j,k}(\tau,\alpha_{1,j,k})\,x_1\,z_k}{h_{x_1}}\\&+ x_{1+1/2}\dfrac{{b_1}_{1+1/2,j,k}(\tau,\alpha_{1,j,k})\,x_{1}^{\beta_{1,j,k}(\tau)} }{h_{x_1}\left(x_{2}^{\beta_{1,j,k}(\tau)}- x_{1}^{\beta_{1,j,k}(\tau)}\right)}, \\
 e_{2,j,k}^{1,j,k} & = -\dfrac{{d_2}_{1,j,k}(\tau,\alpha_{1,j,k})\,x_1\,y_j}{h_{x_1}} - \dfrac{{d_1}_{1,j,k}(\tau,\alpha_{1,j,k})\,x_1\,z_k}{h_{x_1}} \\
 &- x_{1+1/2}\dfrac{{b_1}_{1+1/2,j,k}(\tau,\alpha_{1,j,k})\,x_{2}^{\beta_{1,j,k}(\tau)} }{h_{x_1}\left(x_{2}^{\beta_{1,j,k}(\tau)}- x_{1}^{\beta_{1,j,k}(\tau)}\right)}, \\
 e_{i,0,k}^{i,1,k}& = -\dfrac{1}{2\,y_2 }\,y_1( \bar{a_2}{i,y_{1/2},k}(\tau,\alpha_{i,1,k})- {b_2}_{i,y_{1/2},k}(\tau,\alpha_{i,1,k})) \,v_{i,0,k},\\
 e_{i,1,k}^{i,1,k}& = \dfrac{1}{2\,y_2 } \,y_1 ( \bar{a_2}_{i,y_{1/2},k}(\tau,\alpha_{i,1,k})+ {b_2}_{i,y_{1/2},k}(\tau,\alpha_{i,1,k})) - \dfrac{1}{3}\,c_{i,1,k}(\tau,\alpha_{i,1,k})   \\ 
 &+ \dfrac{{d_1}_{i,1,k}(\tau,\alpha_{i,1,k})\,y_1\,z_k}{h_{y_1}} +  \dfrac{{d_3}_{i,1,k}(\tau,\alpha_{i,1,k})\,x_i\,y_1}{h_{y_1}}\\
 &+ y_{1+1/2}\dfrac{{b_2}_{i,1+1/2,k}(\tau,\alpha_{i,1,k})\,y_{1}^{{\beta_1}_{i,1,k}(\tau)} }{h_{y_1}\left(y_{2}^{{\beta_1}_{i,1,k}(\tau)}- y_{1}^{{\beta_1}_{i,1,k}(\tau)}\right)},
 \end{split}\end{equation}
 \begin{equation}
 \begin{split}
 e_{i,2,k}^{i,1,k}& = -\dfrac{{d_1}_{i,1,k}(\tau,\alpha_{i,1,k})\,y_1\,z_k}{h_{y_1}} - \dfrac{{d_3}_{i,1,k}(\tau,\alpha_{i,1,k})\,x_i\,y_1}{h_{y_1}}\\& - y_{1+1/2}\dfrac{{b_2}_{i,1+1/2,k}(\tau,\alpha_{i,1,k})\,y_{2}^{{\beta_1}_{i,1,k}(\tau)} }{h_{y_1}\left(y_{2}^{{\beta_1}_{i,1,k}(\tau)}- y_{1}^{{\beta_1}_{i,1,k}(\tau)}\right)},
 \end{split}
 \end{equation}
 \begin{equation}
 \begin{split}
 e_{i,j,0}^{i,j,1}& = - \dfrac{1}{2\,z_2 } \,z_{1}(\bar{a_3}_{i,j,z_{1/2}}(\tau,\alpha_{i,j,1})- {b_3}_{i,j,z_{1/2}}(\tau,\alpha_{i,j,1}))\,v_{i,j,0},\\
 e_{i,j,1}^{i,j,1}& =  \dfrac{1}{2\,z_2 } \,z_{1}( \bar{a_3}_{i,j,z_{1/2}}(\tau,\alpha_{i,j,1}) + {b_3}_{i,j,z_{1/2}}(\tau,\alpha_{i,j,1})) \\
 &- \dfrac{1}{3}\,c_{i,j,1}(\tau,\alpha_{i,j,1}) + \dfrac{{d_2}_{i,j,1}(\tau,\alpha_{i,j,1})\,y_j\,z_1}{h_{z_1}}  \\ 
 &+ \dfrac{{d_3}_{i,j,1}(\tau,\alpha_{i,j,1})\,x_i\,z_1}{h_{z_1}}+ z_{1+1/2}\dfrac{{b_3}_{i,j,1+1/2}(\tau,\alpha_{i,j,1})\,z_{1}^{{\beta_2}_{i,j,1}(\tau)} }{h_{z_1}\left(z_{2}^{{\beta_2}_{i,j,1}(\tau)}- z_{1}^{{\beta_2}_{i,j,1}(\tau)}\right)},\\
 e_{i,j,2}^{i,j,1}& = -\dfrac{{d_2}_{i,j,1}(\tau,\alpha_{i,j,1})\,y_j\,z_1}{h_{z_1}} - \dfrac{{d_3}_{i,j,1}(\tau,\alpha_{i,j,1})\,x_i\,z_1}{h_{z_1}}\\& - z_{1+1/2}\dfrac{{{b_3}_{i,j,1+1/2}(\tau,\alpha_{i,j,1})}\,z_{2}^{{\beta_2}_{i,j,1}(\tau)} }{h_{z_1}\left(z_{2}^{{\beta_2}_{i,j,1}(\tau)}- z_{1}^{{\beta_2}_{i,j,1}(\tau)}\right)}.
 \end{split}
 \end{equation}
 $ G $ collects the given homogeneous boundary therm $ v_{0,j,k},\, v_{i,0,k}, \,v_{i,j,0},\, v_{N_1,j,k}, \, v_{i,N_2,k}\, $ \text{and}\,\,  $v_{i,j,N_3} $ for $ i = 1,\cdots, N_1-1 $, \,$ j = 1,\cdots, N_2-1 $\, and\, $ k = 1,\cdots, N_3-1 $. \\
 \begin{theorem} \label{tm}
 	Assume  that the coefficients of $ A $ given  by \eqref{A} are positive and $c<0$ \footnote{Indeed  $c$ can be positive but should be less than a certain threshold $c_0>0$}.
 	Let  \begin{align} h = \underset{{\underset{k = 1,\cdots ,N_3-1}{\underset{j = 1,\cdots ,N_2-1,}{i = 1,\cdots ,N_1-1}}}}{\max} \{h_{x_i},\, h_{y_j},\,\,h_{z_k}\}, \end{align} if $h$ is relatively small then the matrix $E (\tau,\alpha)$  in the system (\ref{p4}) is an $M $-matrix for any $ \alpha_{i,j,k}  \,\in\,\mathcal{A}^{N}$.
 \end{theorem}
{\it Proof}  
 \qed Let us show that $E(\tau,\alpha) $ has positive diagonals, non-positive off-diagonals, and is diagonally dominant.
 	We first note that
 	\begin{equation}
 	\begin{split}
 	& \dfrac{{b_1}_{i+1/2,j,k}(\tau,\alpha_{i,j,k})}{x_{i+1}^{\beta_{i,j,k}(\tau)}- x_{i}^{\beta_{i,j,k}(\tau)}} = \dfrac{\bar{a_1}_{i,j,k}(\tau,\alpha_{i,j,k})\,\beta_{i,j,k}(\tau) }{x_{i+1}^{\beta_{i,j,k}(\tau)}- x_{i}^{\beta_{i,j,k}(\tau)}} > 0,\\ 
 	&\dfrac{{b_2}_{i,j+1/2,k}(\tau,\alpha_{i,j,k})}{y_{j+1}^{{\beta_1}_{i,j,k}(\tau)}- y_{j}^{{\beta_1}_{i,j,k}(\tau)}} = \dfrac{\bar{a_2}_{i,j,k}(\tau,\alpha_{i,j,k})\,{\beta_1}_{i,j,k}(\tau) }{y_{j+1}^{{\beta_1}_{i,j,k}(\tau)}- y_{j}^{{\beta_1}_{i,j,k}(\tau)}} > 0,\\ 
 	& \dfrac{{b_3}_{i,j,k+1/2}(\tau,\alpha_{i,j,k})}{z_{k+1}^{{\beta_2}_{i,j,k}(\tau)}-z_{k}^{{\beta_2}_{i,j,k}(\tau)}} = \dfrac{\bar{a_3}_{i,j,k}(\tau,\alpha_{i,j,k})\,{\beta_2}_{i,j,k}(\tau) }{z_{k+1}^{{\beta_2}_{i,j,k}(\tau)}- z_{k}^{{\beta_2}_{i,j,k}(\tau)}} > 0, 
 	\end{split}
 	\end{equation}
 	for $ i = 1,\cdots, N_1-1 $, $ j = 1,\cdots, N_2-1 $, $ k = 1,\cdots, N_3-1 $ and all $ {b_1}_{i+1/2,j,k}(\tau,\alpha_{i,j,k}) \neq 0,\,\, {b_2}_{i,j+1/2,k}(\tau,\alpha_{i,j,k}) \neq 0, \,\,{b_3}_{i,j,k+1/2}(\tau,\alpha_{i,j,k}) \neq 0$  with $ \bar{a_1}_{i,j,k}(\tau,\alpha_{i,j,k}) > 0 $, 
 	
 	 $ \bar{a_2}_{i,j,k}(\tau,\alpha_{i,j,k}) > 0$ and $ \bar{a_3}_{i,j,k}(\tau,\alpha_{i,j,k}) >0 $.
 	
 	This also holds when $ {b_1}_{i+1/2,j,k}(\tau,\alpha_{i,j,k})\rightarrow 0 $,  $ {b_2}_{i,j+1/2,k}(\tau,\alpha_{i,j,k})\rightarrow 0 $
 	
 	 and $ {b_3}_{i,j,k+1/2}(\tau,\alpha_{i,j,k})\rightarrow 0 $. Indeed
 	\begin{equation*}
 	\begin{split}
 		\lim_{{b_1}_{i+1/2,j,k}(\tau,\alpha)\rightarrow 0}	\dfrac{{b_1}_{i+1/2,j,k}(\tau,\alpha)}{x_{i+1}^{\beta_{i,j,k}(\tau)}- x_{i}^{\beta_{i,j,k}(\tau)}} &= \dfrac{{b_1}_{i+1/2,j,k}(\tau,\alpha) }{e^{\beta_{i,j,k}(\tau)\ln(x_{i+1})}-e^{\beta_{i,j,k}(\tau)\ln(x_{i})}}\\
 		&= \dfrac{{b_1}_{i+1/2,j,k}(\tau,\alpha) }{\beta_{i,j,k}(\tau)\ln(x_{i+1})-\beta_{i,j,k}(\tau)\ln(x_{i})}\\
 		&= \bar{a_1}_{i+1/2,j,k}(\tau,\alpha)\ln \left( \dfrac{ x_{i+1}}{x_{i}}\right)^{-1} > 0,
 		\end{split}
 	\end{equation*} 
 	\begin{equation*}
 	\begin{split}
 		\,\lim_{{b_1}_{i-1/2,j,k}(\tau,\alpha)\rightarrow 0}	\dfrac{{b_1}_{i-1/2,j,k}(\tau)}{x_{i}^{\beta_{i-1,j,k}(\tau)}- x_{i-1}^{\beta_{i-1,j}(\tau)}} &= \dfrac{{b_1}_{i-1/2,j,k}(\tau,\alpha) }{e^{\beta_{i-1,j,k}(\tau)\ln(x_{i})}-e^{\beta_{i-1,j,k}(\tau)\ln(x_{i-1})}}\\
 		& = \dfrac{{b_1}_{i-1/2,j,k}(\tau,\alpha) }{\beta_{i-1,j,k}(\tau)\ln(x_{i})-\beta_{i-1,j}(\tau)\ln(x_{i-1})}\\
 		&= \bar{a_1}_{i-1/2,j,k}(\tau,\alpha) \ln \left( \dfrac{x_{i}}{x_{i-1}}\right)^{-1} > 0,
 		 \end{split}
 	\end{equation*} 
 	Indeed
 	\begin{equation*}
 	\begin{split}
 	&	\lim_{{b_2}_{i,j+1/2,k}(\tau,\alpha)\rightarrow 0}	\dfrac{{b_2}_{i,j+1/2,k}(\tau)}{y_{j+1}^{{\beta_1}_{i,j,k}(\tau)}- y_{j}^{{\beta_1}_{i,j,k}(\tau)}} >0,\\ 
 	&	\lim_{{b_2}_{i,j-1/2,k}(\tau,\alpha)\rightarrow 0}	\dfrac{{b_2}_{i,j-1/2,k}(\tau)}{y_{j}^{{\beta_1}_{i,j-1,k}(\tau)}- y_{j-1}^{{\beta_1}_{i,j-1,k}(\tau)}} >0,\\
 	&	\lim_{{b_3}_{i,j,k+1/2}(\tau,\alpha)\rightarrow 0}	\dfrac{{b_3}_{i,j,k+1/2}(\tau)}{z_{k+1}^{{\beta_2}_{i,j,k}(\tau)}- z_{k}^{{\beta_2}_{i,j,k}(\tau)}} >0,\\
 	&	\lim_{{b_3}_{i,j,k-1/2}(\tau,\alpha)\rightarrow 0}	\dfrac{{b_3}_{i,j,k-1/2}(\tau)}{z_{k}^{{\beta_1}_{i,j,k-1}(\tau)}- z_{k-1}^{{\beta_2}_{i,j,k-1}(\tau)}} >0.
 		\end{split}
 	\end{equation*}	  
 	Using the definition of  $E (\tau,\alpha) =\left(e_{i,j,k}^{i,j,k}\right) $,  $ i= 1,\cdots, N_1-1 $, \,\, $ j = 1,\cdots, N_2-1 $\,\, and\,\, $ k = 1,\cdots, N_3-1 $  given above, we see that
 	\begin{equation*}
 	\begin{split}
 	& e_{i,j,k}^{i,j,k} \geqslant 0,\,\,\,e_{i+1,j,k}^{i,j,k} \leqslant 0,\,\,\, e_{i-1,j,k}^{i,j,k} \leqslant 0\,\,,\,e_{i,j+1,k}^{i,j,k} \leqslant 0, \\ & e_{i,j-1,k}^{i,j,k}\leqslant 0,\,\,\,e_{i,j,k+1}^{i,j,k} \leqslant 0,\,\,\text{and}\,\, e_{i,j,k-1}^{i,j,k}\leqslant 0,
 	\end{split}
 	\end{equation*}
 	For $ i = 2,\cdots, N_1-1 $, $ j = 2,\cdots, N_2-1 $ and $ k = 2,\cdots, N_3-1 $, since 
 	$x_{i+1}^{\beta_{i,j,k}(\tau)} \approx x_{i}^{\beta_{i,j,k}(\tau)} + x_{i}^{\beta_{i,j,k}(\tau)-1}\,{\beta_{i,j,k}(\tau)}\,h_{x_i}$,\,\, $x_{i-1}^{\beta_{i-1,j,k}(\tau)} \approx x_{i}^{\beta_{i-1,j,k}(\tau)} - \\ x_{i}^{\beta_{i-1,j,k}(\tau)-1}\,{\beta_{i-1,j,k}(\tau)}\,h_{x_i}$ \,\, $y_{j+1}^{{\beta_1}_{i,j,k}(\tau)} \approx y_{j}^{{\beta_1}_{i,j,k}(\tau)} + y_{j}^{{\beta_1}_{i,j,k}(\tau)-1}\,{{\beta_1}_{i,j,k}(\tau)}\,h_{y_j}$,\\ $y_{j-1}^{{\beta_1}_{i,j-1,k}(\tau)} \approx y_{j}^{{\beta_1}_{i,j-1,k}(\tau)} - y_{j}^{{\beta_1}_{i,j-1,k}(\tau)-1}\,{{\beta_1}_{i,j-1,k}(\tau)}\,h_{y_j}$ \,\,\\
 	$z_{k+1}^{{\beta_2}_{i,j,k}(\tau)} \approx z_{k}^{{\beta_2}_{i,j,k}(\tau)} + z_{k}^{{\beta_2}_{i,j,k}(\tau)-1}\,{{\beta_2}_{i,j,k}(\tau)}\,h_{z_k}$\,\,and \\ $z_{k-1}^{{\beta_2}_{i,j,k-1}(\tau)} \approx z_{k}^{{\beta_2}_{i,j,k-1}(\tau)} - z_{k}^{{\beta_2}_{i,j,k-1}(\tau)-1}\,{{\beta_2}_{i,j,k-1}(\tau)}\,h_{z_k}$,\\
 	when $ h = \max \{h_{x_i},\, h_{y_j},\,\,h_{z_k}\} \longrightarrow 0 $,
 	\begin{equation*}
 	\begin{split}
 &\left|	e_{i,j,k}^{i,j,k}\right| - \left|e_{i-1,j,k}^{i,j,k} \right|- \left|e_{i,j-1,k}^{i,j,k} \right| - \left|e_{i,j,k-1}^{i,j,k} \right|- \left|e_{i+1,j,k}^{i,j,k}\right|-\left|e_{i,j+1,k}^{i,j,k} \right|-\left|e_{i,j,k}^{i,j,k+1}\right| \\
 &\rightarrow - c_{i,j,k}(\tau,\alpha_{i,j,k}),
 \end{split}
 	\end{equation*}
 	we have 
 	\begin{equation*}
 	\begin{split}
 	\left|e_{i,j,k}^{i,j,k}\right| & \geq \left|e_{i-1,j,k}^{i,j,k}\right|+\left|e_{i,j-1,k}^{i,j,k}\right|+\left|e_{i,j+1,k}^{i,j,k}\right|+\left|e_{i+1,j,k}^{i,j,k}\right|+\left|e_{i,j,k+1}^{i,j,k}\right|+\left|e_{i,j,k-1}^{i,j,k}\right| \\
 	& \geq \sum_{m=1}^{N_1-1} \sum_{n_1=1}^{N_2-1} \sum_{n_2=1}^{N_3-1} \left|e_{m,n_1,n_2}^{i,j,k}\right| ,\,\,m \neq i,\,\,n_1 \neq j, \,\,n_2 \neq k,
 	\end{split}
 	\end{equation*}
 	We also have similar inequalities when one of the indices $i,j,k$ is equal to $1$. 
 	Therefore $ \,E(\tau,\alpha) $ is an ${M} $-matrix.
	
\textcolor{blue}{\section{Fitted finite volume scheme in $ n $ dimensional spatial domain}
The goal here is to update our  three dimension fitted schemes in  high dimensional space ($ n \geq 3 $).
Recall that the HJB equation in $ n\geq 1 $ dimensional space is given by
	\begin{align}\label{generalize}
	\begin{cases}
	v_t(t, x) + \underset{\alpha \in \mathcal{A}}{\sup} \left[L^{\alpha} v(t, x) + f(t, x, \alpha)\right] = 0 \quad\text{on} \,\,   [0,T]\times\mathbb{R}^n ,\\
v(T,x) = g(x), \,\,\,\,x \,\in \mathbb{R}^n
	\end{cases}
	\end{align}
\begin{align*}
\text{where}\,\,\,L^\alpha\,v(t, x) = \dfrac{1}{2}\,\sum_{i,j = 1}^n (\sigma \sigma^T)_{i,j} (t, x, \alpha)\dfrac{\partial^2 v(t, x)}{\partial x_i\,\partial x_j}+ \sum_{i = 1}^n b_i(t, x,\alpha)\dfrac{\partial v(t, x)}{ \partial x_i}.
\end{align*}The divergence form of equation \eqref{generalize}
 by setting $ \tau = T-t $ is given by
\begin{align} \label{vIe1}
-\dfrac{\partial v(\tau, x) }{\partial \tau} + \sup_{\alpha \in \mathcal{A}}\left[ \nabla\cdot \left( k (v(\tau, x))\right) + c(\tau, x, \alpha)\,v(\tau, x) \right] = 0,
\end{align}
where $k(v(\tau, x)) = A(\tau, x,\alpha)\nabla v(\tau, x)+ b(\tau, x, \alpha)\,v(\tau, x)$\,  with   
\begin{eqnarray}
\label{matrixA}
b = (x_1\,b_1, x_2\,b_2, x_3\,b_3,\cdots, x_n\,b_n)^T,\quad
A=\left[ \begin{array}{ccccc}
a_{11} & a_{12} & a_{13}& \cdots  & a_{1n} \\
a_{21} & a_{22}& a_{23} & \cdots  & a_{2n}  \\
a_{31} & a_{32}& a_{33} & \cdots  & a_{3n}\\
\vdots & \vdots  & \vdots & \vdots  & \vdots\\
a_{n1} & a_{n2}& a_{n3} & \cdots  & a_{nn}
\end{array} \right].
\end{eqnarray}
Indeed this  divergence form is not a restriction as the differentiation is respect to $x$ 
and not respect to the control  $\alpha$, which may be discontinuous in some applications. We will assume that for $ i\neq r,\,\,a_{ir} = a_{ri},\,\, \,\,r,i =1,\cdots,n$.  We also define the following coefficients, which will help us to build our scheme smoothly 
\begin{align*} 
& a_{ii}(\tau, x, \alpha) = \overline{a_i}(\tau, x, \alpha)\,x_i^2\,\, 
 \text{and}\,\,  a_{ir} = a_{ri} = d_{ir}(\tau, x, \alpha) \prod_{i=1}^{n} x_i=d_{ri}(\tau, x, \alpha) \prod_{i=1}^{n} x_i,\,\,\,r\neq i, 
\end{align*}
 $i,r=1,\cdots,n. $
 As usual the $ n $ dimensional domain is truncated to $ I_{x_i}= [0, {x_i}_{\text{max}}] $,\, $ i = 1,\cdots,n $  be divided into $ N_i $ sub-intervals
 $$  I_{{x_1}_j} =({x_1}_j, {x_1}_{j+1}),\,\, I_{{x_2}_k} =({x_2}_k, {x_2}_{k+1}) ,\,\, I_{{x_3}_l} =({x_3}_l, {x_3}_{l+1}),\cdots,I_{{x_n}_m} =({x_n}_m, {x_n}_{m+1})$$
 $ j = 0\cdots N_1-1,\,\,k = 0\cdots N_2-1,\,\,\,l=0\cdots N_3-1,\cdots, m=0\cdots N_n-1,$
 with $ 0 = {x_i}_{0} < {x_i}_{1}< \cdots \cdots< {x_i}_{p} = {x_i}_{\text{max}},\,$.
 This defines on $ I_{x} = \overset{n}{\underset{i=1}{\prod}} I_{x_i} $  a rectangular  mesh.
 By setting 
 \begin{equation}
 	\begin{split}
 & {x_1}_{j+1/2} :=\dfrac{{x_1}_{j} + {x_1}_{j+1} }{2},\qquad {x_1}_{j-1/2} :=\dfrac{{x_1}_{j} + {x_1}_{j-1} }{2},\\  
  & {x_2}_{k+1/2} :=\dfrac{{x_2}_{k} + {x_2}_{k+1} }{2},\qquad {x_2}_{k-1/2} :=\dfrac{{x_2}_{k} + {x_2}_{k-1} }{2},\\
  & {x_3}_{l+1/2} :=\dfrac{{x_3}_{l} + {x_3}_{l+1} }{2},\,\, \,\qquad {x_3}_{l-1/2} :=\dfrac{{x_3}_{l} + {x_3}_{l-1} }{2},\\
  & \qquad \vdots \qquad\qquad\qquad\qquad \qquad\qquad\quad\vdots\\
  & {x_n}_{m+1/2} :=\dfrac{{x_n}_{m} + {x_n}_{m+1} }{2},\,\,\,\, {x_n}_{m-1/2} :=\dfrac{{x_n}_{m} + {x_n}_{m-1} }{2},\\
 \end{split}
 \end{equation}
 for each $ j=1\cdots N_1-1$,\,\, $ k=1\cdots N_2-1$, $ l=1\cdots N_3-1,\cdots, m=1\cdots N_n-1 $. These mid-points form a second partition of $ I_{x} = \overset{n}{\underset{i=1}{\prod}} I_{x_i} $ if we define $ {x_i}_{-1/2} = {x_i}_{0}$,\, $ {x_i}_{N_i+1/2} = {x_i}_{\text{max}}$,\,\,\,$ i=1,2,\cdots, n $. For each $j = 0, 1, \cdots ,N_1 $,\,\, $k = 0, 1, \cdots ,N_2 $, $l = 0, 1,\cdots, N_3,\cdots, m = 0, 1,\cdots,N_n $ \, we put\,  $h_{{x_1}_j} = {x_1}_{j+1/2} - {x_1}_{j-1/2} $, $h_{{x_2}_k} = {x_2}_{k+1/2} - {x_2}_{k-1/2} $, \, $h_{{x_3}_l} = {x_3}_{l+1/2} - {x_3}_{l-1/2} $, $\cdots, h_{{x_n}_m} = {x_n}_{m+1/2} - {x_n}_{m-1/2} $  and\\ $ h = \max \{h_{{x_1}_j},\,h_{{x_2}_k},\,h_{{x_3}_l},\cdots, h_{{x_n}_m}\} $.
 Integrating both size of (\ref{generalize}) over $ \mathcal{R}_{j,k,l,\cdots,m} = \left[{x_1}_{j-1/2}, {x_1}_{j+1/2} \right] \times \left[{x_2}_{k-1/2}, {x_2}_{k+1/2} \right]\times \left[{x_3}_{l-1/2}, {x_3}_{l+1/2} \right]\times \cdots\times \left[{x_n}_{m-1/2}, {x_n}_{m+1/2} \right] $ we have 
 \begin{equation}
 \begin{split}
 &- \int_{\mathcal{R}_{j,k,l,\cdots,m}}  \dfrac{ \partial v }{\partial \tau}\, dx_1\,dx_2\, dx_3\cdots dx_n \\
 & + \int_{\mathcal{R}_{j,k,l,\cdots,m}} \sup_{\alpha \in \mathcal{A}} \left[ \nabla \cdot \left( k(v)\right) + c\,v \right]\, dx_1\,dx_2\, dx_3\cdots dx_n = 0,
 \end{split}
 \end{equation}
 for $ j =1,2,\cdots N_1-1 $,\, $ k =1,2,\cdots N_2-1 $,\, $l =1,2,\cdots N_3-1,\cdots,m =1,2,\cdots N_n-1 $.\\
 Applying the mid-points quadrature rule to the first and the last point terms, we obtain the above
 \begin{equation}\label{Bjr1}
 \begin{split}
 &-\dfrac{d v_{j,k,l,\cdots,m}(\tau) }{d \tau}\,l_{j,k,l,\cdots,m}+ \\
 & \sup_{\alpha \in \mathcal{A}} \left[ \int_{ \mathcal{R}_{j,k,l,\cdots,m}}\nabla \cdot \left( k(v)\right)\,dx_1\,dx_2\cdots dx_n + c_{j,k,l,\cdots,m}(\tau, \alpha)\, v_{j,k,l,\cdots,m}(\tau)\,l_{j,k,l,\cdots,m}\right] =0
 \end{split}
 \end{equation}
where $ l_{j,k,l,\cdots,m} = \left( {x_1}_{j+1/2} - {x_1}_{j-1/2} \right) \times \left({x_2}_{k+1/2} - {x_2}_{k-1/2}\right) \times \left({x_3}_{l+1/2} - {x_3}_{l-1/2}\right)\times\cdots\times \left({x_n}_{m+1/2} - {x_n}_{m-1/2}\right) $ is the  volume of $ \mathcal{R}_{j,k,l,\cdots,m} $. Note that $ v_{j,k,l,\cdots,m}(\tau) $ denotes the nodal approximation to $ v(\tau, {x_1}_{j}, {x_2}_{k}, {x_3}_{l},\cdots,{x_n}_{m}) $ at each point of the grid.\\
 We now consider the approximation of the middle term in \eqref{Bjr1}. Let $\bf n $ denote the unit vector outward-normal to $ \partial \mathcal{R}_{j,k,l,\cdots,m} $. By General Stokes Theorem, integrating by parts and using the definition
 of flux $ k $, we have
 \begin{equation}\label{Vol}
 \begin{split}
 & \int_{\mathcal{R}_{j,k,l,\cdots,m}} \nabla \cdot \left( k(v)\right) d x_1\,dx_2\,dx_3\cdots dx_n \\ \nonumber
 & = \int_{\partial \mathcal{R}_{j,k,l,\cdots,m}} k(v) \cdot \bf n \, ds \\ 
 &= \int_{\left({x_1}_{j+1/2},{x_2}_{k-1/2},{x_3}_{l-1/2},\cdots,{x_n}_{m-1/2} \right)}^{\left({x_1}_{j+1/2},{x_2}_{k+1/2},{x_3}_{l+1/2},\cdots,{x_n}_{m+1/2} \right)}\left(\sum_{i = 1}^{n}a_{1i}\,\dfrac{\partial v }{\partial x_i}+ x_1\,b_1\,v \right)dx_2\,dx_3\cdots dx_n  \\ 
 &- \int_{\left({x_1}_{j-1/2},{x_2}_{k-1/2},{x_3}_{l-1/2},\cdots,{x_n}_{m-1/2} \right)}^{\left({x_1}_{j-1/2},{x_2}_{k+1/2},{x_3}_{l+1/2},\cdots,{x_n}_{m+1/2} \right)} \left(\sum_{i = 1}^{n}a_{1i}\,\dfrac{ \partial v }{\partial x_i}+ x_1\,b_1\,v \right)dx_2\,dx_3\cdots dx_n  \\ 
 & +\int_{\left({x_1}_{j-1/2},{x_2}_{k+1/2},{x_3}_{l-1/2},\cdots,{x_n}_{m-1/2} \right)}^{\left({x_1}_{j+1/2},{x_2}_{k+1/2},{x_3}_{l+1/2},\cdots,{x_n}_{m+1/2} \right)} \left(\sum_{i=1}^{n}a_{2i}\,\dfrac{\partial v }{ \partial x_i} + x_2\,b_2\,v \right)dx_1\,dx_3\cdots dx_n  \\ 
 &- \int_{\left({x_1}_{j-1/2},{x_2}_{k-1/2},{x_3}_{l-1/2},\cdots,{x_n}_{m-1/2} \right)}^{\left({x_1}_{j+1/2},{x_2}_{k-1/2},{x_3}_{l+1/2},\cdots,{x_n}_{m+1/2} \right)} \left(\sum_{i=1}^{n}a_{2i}\,\dfrac{\partial  v }{\partial x_i}+  x_2\,b_2\,v \right)dx_1\,dx_3\cdots dx_n  \\ 
 & \quad\vdots \quad\quad\quad\quad\vdots \quad\quad\quad\quad\vdots\quad\quad\quad\vdots\quad\quad\quad\vdots\quad\quad\quad\vdots\quad\quad\quad\vdots\quad\quad\quad\vdots\quad\quad\quad\vdots\quad\quad\quad\vdots\quad\quad\quad\vdots\\
 &+ \int_{\left({x_1}_{j-1/2},{x_2}_{k-1/2},{x_3}_{l-1/2},\cdots,{x_n}_{m+1/2} \right)}^{\left({x_1}_{j+1/2},{x_2}_{k+1/2},{x_3}_{l+1/2},\cdots,{x_n}_{m+1/2} \right)}  \left(\sum_{i=1}^{n}a_{ni}\,\dfrac{\partial v }{\partial x_i} + x_n\,b_n\,v \right)dx_1\,dx_2\,dx_3\cdots dx_{n-1}\\
 & - \int_{\left({x_1}_{j-1/2},{x_2}_{k-1/2},{x_3}_{l-1/2},\cdots,{x_n}_{m-1/2} \right)}^{\left({x_1}_{j+1/2},{x_2}_{k+1/2},{x_3}_{l+1/2},\cdots,{x_n}_{m-1/2} \right)} \left(\sum_{i=1}^{n}a_{ni}\,\dfrac{\partial v }{\partial x_i} + x_n\,b_n\,v \right)dx_1\,dx_2\,dx_3\cdots dx_{n-1}\\
 & = \sum_{i =1}^{n}\bigg(\int_{\left({x_1}_{j-1/2},{x_2}_{k-1/2},{x_3}_{l-1/2},\cdots,{x_i}_{q+1/2},\cdots,{x_n}_{m-1/2} \right)}^{\left({x_1}_{j+1/2},{x_2}_{k+1/2},{x_3}_{l+1/2},\cdots,{x_i}_{q+1/2},\cdots,{x_n}_{m+1/2} \right)} \left( \sum_{r =1}^{n} a_{ir}\,\dfrac{\partial v }{\partial x_r} + x_i\,b_i\,v \right)\prod_{i\neq r}^{n}dx_r\bigg)\\
 & - \sum_{i =1}^{n}\bigg(\int_{\left({x_1}_{j-1/2},{x_2}_{k-1/2},{x_3}_{l-1/2},\cdots,{x_i}_{q-1/2},\cdots,{x_n}_{m-1/2} \right)}^{\left({x_1}_{j+1/2},{x_2}_{k+1/2},{x_3}_{l+1/2},\cdots,{x_i}_{q-1/2},\cdots,{x_n}_{m+1/2} \right)}\left(\sum_{r =1}^{n}  a_{ir}\,\dfrac{\partial v }{\partial x_r} + x_i\,b_i\,v \right)\prod_{i\neq r}^{n}dx_r\bigg)
 \end{split} 
 \end{equation} 
 We will approximate the first term using the the mid-points quadrature rule as 
 \begin{equation}
 \begin{split}
 & \sum_{i =1}^{n}\bigg(\int_{\left({x_1}_{j-1/2},{x_2}_{k-1/2},{x_3}_{l-1/2},{x_i}_{q+1/2},\cdots,{x_n}_{m-1/2} \right)}^{\left({x_1}_{j+1/2},{x_2}_{k+1/2},{x_3}_{l+1/2},{x_i}_{q+1/2},\cdots,{x_n}_{m+1/2} \right)} \left( \sum_{r =1}^{n} a_{ir}\,\dfrac{\partial v }{\partial x_r} + x_i\,b_i\,v \right)\prod_{i\neq r}^{n}dx_r\bigg)\\& \approx \sum_{i,r =1}^{n}\left(a_{ir}\,\dfrac{\partial v }{\partial x_r} + x_i\,b_i\,v \right)\bigg|_{\left({x_1}_{j},{x_2}_{k},{x_3}_{l},\cdots,{x_i}_{q+1/2},{x_i}_{s},\cdots,{x_n}_{m}\right)} \prod_{i\neq r}^{n}h_{{x_r}_\nu}.
 \end{split}
 \end{equation} where the value of the subscript $ \nu \in \{j,k,l,\cdots,q,s,\cdots,m\} $  depends respectively of the value taking by $  r \in \{1,2,3,\cdots,i, \cdots,n\}  $.
 To achieve this, it is clear that we now need to derive approximations of the  $ k(v) \cdot \bf n $ defined above at the mid-point \\$ \left({x_1}_{j},{x_2}_{k},{x_3}_{l},\cdots,{x_i}_{q+1/2},{x_{i+1}}_{s},\cdots,{x_n}_{m}\right) $, of the interval $ I_{{x_i}_q} $ for $ q = 0, 1,\cdots N_i-1$,\, $i=1,2,\cdots,n $. This discussion is divided into two cases for $ q \geq 1 $, and $ q = 0\, $ on the interval $ I_{{x_i}_0} = [0, {x_i}_1],\,\, i=1,2,\cdots,n  $. This is really the  generalization of the fitted finite scheme.\\ 
 \textbf{\underline{Case I}:} For $ q\geq 1 $.\\
 We follow the same procedure as in three dimension and  have the following generalization
 	\begin{equation}\label{Vol1}
 	\begin{split}
 	&\sum_{i,r =1}^{n}\left(a_{ir}\,\dfrac{\partial v }{ \partial x_r} + x_i\,b_i\,v \right)\bigg|_{\left({x_1}_{j},{x_2}_{k},{x_3}_{l},\cdots,{x_i}_{q+1/2},{x_{i+1}}_{s},\cdots,{x_n}_{m}\right)}    \prod_{i\neq r}^{n} h_{{x_r}_\nu}\\
 	& \approx \sum_{i=1}^{n}{x_i}_{q+1/2} {b_i}_{j,k,l,\cdots,q+1/2,s\cdots,m}(\tau, \alpha_{j,k,l,\cdots,q,s,\cdots,m})\times\\
 	&\dfrac{\left({x_i}_{q+1}^{\beta_{j,k,l,\cdots,q,s\cdots,m}(\tau)}\,v_{j,k,l,\cdots,q+1,s,\cdots,m}-{x_i}_{q}^{\beta_{j,k,l,\cdots,q,s\cdots,m}(\tau)}\,v_{j,k,l,\cdots,q,s,\cdots,m} \right)}{{x_i}_{q+1}^{\beta_{j,k,l,\cdots,q,s\cdots,m}(\tau)}- {x_i}_{q}^{\beta_{j,k,l,\cdots,q,s\cdots,m}(\tau)}}\prod_{i\neq r}^{n} h_{{x_r}_\nu}\\ 
 	&+\sum_{\underset{i\neq r}{i,r=1}}^n{x_i}_{q+1/2}\left( {d_{ir}}_{j,k,l,\cdots,q,s,\cdots,m}(\tau,\alpha_{j,k,l,\cdots,q,s,\cdots,m})\left(\prod_{i\neq r}{x_r}_\nu\right)\times\right.\\
 	&\left.\dfrac{v_{j,k,l,\cdots,q,s,\nu+1,\cdots,m}-v_{j,k,l,\cdots,q,s,\nu,\cdots,m}}{ h_{{x_r}_\nu}}  \right)\prod_{i\neq r}^{n} h_{{x_r}_\nu}.
 	\end{split}
 	\end{equation}
 	where the value of the subscript $ \nu \in \{j,k,l,\cdots,q,s,\cdots,m\} $  depends respectively of the value taking by $ r \in \{1,2,3,\cdots,i,\cdots,n\} $,\\ $ \beta_{j,k,l,\cdots,q,s,\nu\cdots,m}(\tau) = \dfrac{{b_i}_{j,k,l,\cdots,q+1/2,s,r\cdots,m}(\tau, \alpha_{j,k,l,\cdots,q,s,\nu,\cdots,m})}{{{\overline{a}_i}}_{j,k,l,\cdots,q+1/2,s,\nu\cdots,m}(\tau, \alpha_{j,k,l,\cdots,q,s,\nu,\cdots,m})}\neq 0.$ Similarly 
 	\begin{equation}\label{Vol2}
 	\begin{split}
 	&\sum_{i,r =1}^{n}\left(a_{ir}\,\dfrac{\partial  v }{\partial x_r} + x_i\,b_i\,v \right)\bigg|_{\left({x_1}_{j},{x_2}_{k},{x_3}_{l},\cdots,{x_i}_{q-1/2},{x_{i+1}}_{s},\cdots,{x_n}_{m}\right)}    \prod_{i\neq r}^{n}dx_r\\
 	& \approx \sum_{i=1}^{n}{x_i}_{q-1/2} {b_i}_{j,k,l,\cdots,q-1/2,s\cdots,m}(\tau, \alpha_{j,k,l,\cdots,q,s,\cdots,m})\times\left(\prod_{i\neq r}^{n} h_{{x_r}_\nu}\right)\times\\
 	&\dfrac{\left({x_i}_{q}^{\beta_{j,k,l,\cdots,q-1,s\cdots,m}(\tau)}\,v_{j,k,l,\cdots,q,s,\cdots,m}-{x_i}_{q-1}^{\beta_{j,k,l,\cdots,q-1,s\cdots,m}(\tau)}\,v_{j,k,l,\cdots,q-1,s,\cdots,m} \right)}{{x_i}_{q}^{\beta_{j,k,l,\cdots,q-1,s\cdots,m}(\tau)}- {x_i}_{q-1}^{\beta_{j,k,l,\cdots,q-1,s\cdots,m}(\tau)}}\\ 
 	&+\sum_{\underset{i\neq r}{i,r=1}}^n{x_i}_{q-1/2}\left( {d_{ir}}_{j,k,l,\cdots,q,s,\cdots,m}(\tau,\alpha_{j,k,l,\cdots,q,s,\cdots,m})\left(\prod_{i\neq r}{x_r}_\nu\right)\times\right.\\
 	&\left.\dfrac{v_{j,k,l,\cdots,q,s,\nu+1,\cdots,m}-v_{j,k,l,\cdots,q,s,\nu,\cdots,m}}{ h_{{x_r}_\nu}}  \right)\prod_{i\neq r}^{n} h_{{x_r}_\nu}.
 	\end{split}
 	\end{equation}
 	where  $ \beta_{j,k,l,\cdots,q-1,s,\nu\cdots,m}(\tau) = \dfrac{{b_i}_{j,k,l,\cdots,q-1/2,s,\nu,\cdots,m}(\tau, \alpha_{j,k,l,\cdots,q,s,\nu,\cdots,m})}{{{\overline{a}_i}}_{j,k,l,\cdots,q-1/2,s,\nu,\cdots,m}(\tau, \alpha_{j,k,l,\cdots,q,s,\nu,\cdots,m})}\neq 0.$\\ 	
 	\textbf{\underline{Case II:}} Approximation of the flux at $ q = 0 $ on the interval $ I_{{x_i}_0} = [0, {x_i}_1],\,\,i=1,\cdots,n$.
 	Note that the analysis in case I does not apply to the approximation of the flux on $ I_{{x_i}_0} $ because it is the degenerated zone. Follow the same lines as for  the three dimensional case, we get
 	\begin{equation}
 	\begin{split}
 	&\sum_{i,r =1}^{n}\left(a_{ir}\,\dfrac{\partial v }{\partial x_r} + x_i\,b_i\,v \right)\bigg|_{\left({x_1}_{j},{x_2}_{k},{x_3}_{l},\cdots,{x_i}_{1/2},{x_i}_{s},\cdots,{x_n}_{m}\right)} \prod_{i\neq r}^{n}dx_r\\
 	& \approx\sum_{i=1}^n {x_i}_{1/2} \dfrac{1}{2}\bigg( \bigg({a_i}_{j,k,l,\cdots,1/2,s\cdots,m}(\tau, \alpha_{j,k,l,\cdots,1,s,\cdots,m})\\
 	&+{b_i}_{j,k,l,\cdots,1/2,s\cdots,m}(\tau, \alpha_{j,k,l,\cdots,1,s,\cdots,m})\bigg)\,v_{j,k,l,\cdots,1,s,\cdots,m} \prod_{i\neq r}^{n} h_{{x_r}_\nu}\bigg) \\  
 	&  -\sum_{i=1}^n {x_i}_{1/2}\dfrac{1}{2}\bigg(\bigg({a_i}_{j,k,l,\cdots,1/2,s\cdots,m}(\tau, \alpha_{j,k,l,\cdots,1,s,\cdots,m})\\
 	&-{b_i}_{j,k,l,\cdots,1/2,s\cdots,m}(\tau, \alpha_{j,k,l,\cdots,1,s,\cdots,m})\bigg)\,v_{j,k,l,\cdots,0,s,\cdots,m}\prod_{i\neq r}^{n} h_{{x_r}_\nu} \bigg)\\ 
 	&+ \sum_{\underset{i\neq r}{i,r=1}}^n{x_i}_{1/2}\left( {d_{ir}}_{j,k,l,\cdots,1,s,\cdots,m}(\tau,\alpha_{j,k,l,\cdots,1,s,\cdots,m})\left(\prod_{i\neq r}{x_r}_\nu\right)\times\right.\\
 	&\left.\dfrac{v_{j,k,l,\cdots,1,s,\nu+1,\cdots,m}-v_{j,k,l,\cdots,1,s,\nu,\cdots,m}}{ h_{{x_r}_\nu}}  \right)\prod_{i\neq r}^{n} h_{{x_r}_\nu}.
 	\end{split}
 	\end{equation}
 	Equation (\ref{Bjr1}) becomes by replacing the flux by its value for $ j = 1,\cdots,N_1-1 $, \,$ k = 1,\cdots,N_2-1 $,  \, $ l = 1,\cdots,N_3-1,\cdots, m = 1,\cdots,N_n-1, $
	 and $ N = \prod_{i=1}^{n} (N_i-1)$.
 	\begin{equation}
 	\begin{split}
 	&-\dfrac{d\, v_{j,k,l,\cdots,q,s,\nu,\cdots,m}(\tau)(\tau)}{d\, \tau}+\dfrac{1}{l_{j,k,l,\cdots,q,s,\cdots,m}} \times\\ 
 	& \underset{\alpha_{j,k,l,\cdots,q,s,\cdots,m} \in \mathcal{A}^{N}}{\sup} \bigg[\sum_{i=1}^{n}{x_i}_{q+1/2} {b_i}_{j,k,l,\cdots,q+1/2,s\cdots,m}(\tau, \alpha_{j,k,l,\cdots,q,s,\cdots,m})\times\\
 	&\dfrac{\left({x_i}_{q+1}^{\beta_{j,k,l,\cdots,q,s\cdots,m}(\tau)}\,v_{j,k,l,\cdots,q+1,s,\cdots,m}-{x_i}_{q}^{\beta_{j,k,l,\cdots,q,s\cdots,m}(\tau)}\,v_{j,k,l,\cdots,q,s,\cdots,m} \right)}{{x_i}_{q+1}^{\beta_{j,k,l,\cdots,q,s\cdots,m}(\tau)}- {x_i}_{q}^{\beta_{j,k,l,\cdots,q,s\cdots,m}(\tau)}}\prod_{i\neq r}^{n} h_{{x_r}_\nu}\\ 
 	&+\sum_{\underset{i\neq r}{i,r=1}}^n{x_i}_{q+1/2}\bigg( {d_{ir}}_{j,k,l,\cdots,q,s,\cdots,m}(\tau,\alpha_{j,k,l,\cdots,q,s,\cdots,m})\left(\prod_{i\neq r}{x_r}_\nu\right)\times\\
 	&\dfrac{v_{j,k,l,\cdots,q,s,\nu+1,\cdots,m}-v_{j,k,l,\cdots,q,s,\nu,\cdots,m}}{ h_{{x_r}_\nu}}  \bigg)\prod_{i\neq r}^{n} h_{{x_r}_\nu}\\
 	& - \sum_{i=1}^{n}{x_i}_{q-1/2} {b_i}_{j,k,l,\cdots,q-1/2,s\cdots,m}(\tau, \alpha_{j,k,l,\cdots,q,s,\cdots,m})\times\left(\prod_{i\neq r}^{n} h_{{x_r}_\nu}\right)\times\\
 	&\dfrac{\left({x_i}_{q}^{\beta_{j,k,l,\cdots,q-1,s\cdots,m}(\tau)}\,v_{j,k,l,\cdots,q,s,\cdots,m}-{x_i}_{q-1}^{\beta_{j,k,l,\cdots,q-1,s\cdots,m}(\tau)}\,v_{j,k,l,\cdots,q-1,s,\cdots,m} \right)}{{x_i}_{q}^{\beta_{j,k,l,\cdots,q-1,s\cdots,m}(\tau)}- {x_i}_{q-1}^{\beta_{j,k,l,\cdots,q-1,s\cdots,m}(\tau)}}\\ 
 	&+\sum_{\underset{i\neq r}{i,r=1}}^n{x_i}_{q-1/2}\bigg( {d_{ir}}_{j,k,l,\cdots,q,s,\cdots,m}(\tau,\alpha_{j,k,l,\cdots,q,s,\cdots,m})\left(\prod_{i\neq r}{x_r}_\nu\right)\times\\
 	&\dfrac{v_{j,k,l,\cdots,q,s,\nu+1,\cdots,m}-v_{j,k,l,\cdots,q,s,\nu,\cdots,m}}{ h_{{x_r}_\nu}}  \bigg)\prod_{i\neq r}^{n} h_{{x_r}_\nu}\\
 	& + c_{j,k,l,\cdots,q,s,\nu,\cdots,m}\,v_{j,k,l,\cdots,q,s,\nu,\cdots,m}\,l_{j,k,l,\cdots,q,s,\nu,\cdots,m}\bigg]=0.
 	\end{split} 
 	\end{equation}	
 	This can be rewritten as the Ordinary Differential Equation (ODE) coupled with optimization
 			\begin{equation}\label{stem}
 			\begin{cases}
 			\dfrac{d\,\textbf{v}(\tau)}{d\,\tau} + \underset{\alpha \in \mathcal{A}^N}{\inf}\,\left[E (\tau, \alpha)\,\textbf{v}(\tau) + F(\tau,\alpha) \right] = 0, \\
 			~~~\mbox{with}~~ \,\,\,\,  \textbf{v}(0)\,\,\,\text{given},
 			\end{cases}
 			\end{equation} 
 			where  $  E (\tau,\alpha)$  is an $N\times N$ matrix,  \,$ \mathcal{A}^N=\underset{N}{\underbrace{\mathcal{A}\times\cdots\times\mathcal{A}}} $,\, $G (\tau, \alpha) =-F(\tau, \alpha) $ depends of  the boundary condition  and the term $c$.
 			$\textbf{v}= \left( v_{j,k,l,\cdots,q,s,\nu,\cdots,m}\right) $, 
 			\,\, and  \,\, $G (\tau, \alpha) =-F(\tau, \alpha) $.
 			By setting
 			$n_1=N_1-1,\; n_2=N_2-1; \;n_3=N_3-1;,\cdots,n_n=N_n-1;, \;\; I:=I (j,k,l,\cdots,q,s,\nu,\cdots,m)= j + (k-1)n_1 +(l-1)n_1 n_2 + \cdots + (m-1)\overset{n-1}{\underset{i=1}{\prod}}n_i$ and  $J:=J (j',k',l',\cdots,q',s',\nu',\cdots,m')=  j' + (k'-1)n_1 +(l'-1)n_1 n_2 + \cdots + (m'-1)\overset{n-1}{\underset{i=1}{\prod}}n_i $, we have 
 			$E (\tau, \alpha) (I,J)= \left(e_{j',k',l'\cdots,q',s',\nu',\cdots,m'}^{j,k,l,\cdots,q,s,\nu,\cdots,m}\right) $,  $j', j = 1,\cdots, N_1-1 $, \,\, $ k',k = 1,\cdots, N_2-1 $\,\, and\,\, $l', l = 1,\cdots, N_3-1,\cdots,m',m = 1,\cdots, N_n-1 $ and
 			\begin{equation}
 			\begin{split}
 			&e_{j,k,l\cdots,q,s,r,\cdots,m}^{j,k,l,\cdots,q,s,\nu,\cdots,m}\\
 			&=\dfrac{1}{l_{j,k,l,\cdots,q,s,\cdots,m}} \bigg[\sum_{i=1}^{n}{x_i}_{q+1/2} {b_i}_{j,k,l,\cdots,q+1/2,s\cdots,m}(\tau, \alpha_{j,k,l,\cdots,q,s,\cdots,m})\times\\
 			&\dfrac{{x_i}_{q}^{\beta_{j,k,l,\cdots,q,s\cdots,m}(\tau)} }{{x_i}_{q+1}^{\beta_{j,k,l,\cdots,q,s\cdots,m}(\tau)}- {x_i}_{q}^{\beta_{j,k,l,\cdots,q,s\cdots,m}(\tau)}}\prod_{i\neq r}^{n} h_{{x_r}_\nu}\\ 
 			&+\sum_{\underset{i\neq r}{i,r=1}}^n\bigg( {d_{ir}}_{j,k,l,\cdots,q,s,\cdots,m}(\tau,\alpha_{j,k,l,\cdots,q,s,\cdots,m})\left(\prod_{i\neq r}{x_r}_\nu\right)\dfrac{1}{h_{{x_r}_\nu}}
 			  \bigg)\prod_{i=1}^{n} h_{{x_i}_q}\\
 			&+\sum_{i=1}^{n}{x_i}_{q-1/2} {b_i}_{j,k,l,\cdots,q-1/2,s\cdots,m}(\tau, \alpha_{j,k,l,\cdots,q,s,\cdots,m})\times\left(\prod_{i\neq r}^{n} h_{{x_r}_\nu}\right)\times\\
 			&\dfrac{{x_i}_{q}^{\beta_{j,k,l,\cdots,q-1,s\cdots,m}(\tau)}}{{x_i}_{q}^{\beta_{j,k,l,\cdots,q-1,s\cdots,m}(\tau)}- {x_i}_{q-1}^{\beta_{j,k,l,\cdots,q-1,s\cdots,m}(\tau)}}\bigg]
 			- c_{j,k,l,\cdots,q,s,\nu,\cdots,m},\\
 			& \vspace{0.5cm}\\
 			&\sum_{\overset{j'\neq j}{j'=j-1}}^{N_1-1} \sum_{\overset{k'\neq k}{k'=k-1}}^{N_2-1} \sum_{\overset{l'\neq l}{l'=l-1}}^{N_3-1}\cdots \sum_{\overset{m'\neq m}{m'=m-1}}^{N_n-1} e_{j',k',l'\cdots,q',s',\nu',\cdots,m'}^{j,k,l,\cdots,q,s,r,\cdots,m}\\
 			& =\dfrac{1}{l_{j,k,l,\cdots,q,s,\cdots,m}} \bigg[
 			 \sum_{i=1}^{n}{x_i}_{q-1/2} {b_i}_{j,k,l,\cdots,q-1/2,s\cdots,m}(\tau, \alpha_{j,k,l,\cdots,q,s,\cdots,m})\times\\
 			&\dfrac{-{x_i}_{q-1}^{\beta_{j,k,l,\cdots,q-1,s\cdots,m}(\tau)}}{{x_i}_{q}^{\beta_{j,k,l,\cdots,q-1,s\cdots,m}(\tau)}- {x_i}_{q-1}^{\beta_{j,k,l,\cdots,q-1,s\cdots,m}(\tau)}}\left(\prod_{i\neq r}^{n} h_{{x_r}_\nu}\right)\bigg],\\
 			& \vspace{0.5cm}\\
 			&\sum_{\overset{j'\neq j}{j'=j+1}}^{N_1-1} \sum_{\overset{k'\neq k}{k'=k+1}}^{N_2-1} \sum_{\overset{l'\neq l}{l'=l+1}}^{N_3-1}\cdots \sum_{\overset{m'\neq m}{m'=m+1}}^{N_n-1} e_{j',k',l'\cdots,q',s',\nu',\cdots,m'}^{j,k,l,\cdots,q,s,\nu,\cdots,m}\\
 			& =\dfrac{1}{l_{j,k,l,\cdots,q,s,\cdots,m}} \bigg[
 			\sum_{i=1}^{n}{x_i}_{q+1/2} {b_i}_{j,k,l,\cdots,q+1/2,s\cdots,m}(\tau, \alpha_{j,k,l,\cdots,q,s,\cdots,m})\times\\
 			&\dfrac{-{x_i}_{q+1}^{\beta_{j,k,l,\cdots,q,s\cdots,m}(\tau)}}{{x_i}_{q+1}^{\beta_{j,k,l,\cdots,q,s\cdots,m}(\tau)}- {x_i}_{q}^{\beta_{j,k,l,\cdots,q,s\cdots,m}(\tau)}}\prod_{i\neq r}^{n} h_{{x_r}_\nu}\\ 
 			&-\sum_{\underset{i\neq r}{i,r=1}}^n\bigg( {d_{ir}}_{j,k,l,\cdots,q,s,\nu,\cdots,m}(\tau,\alpha_{j,k,l,\cdots,q,s,\nu,\cdots,m})\left({x_i}_q {x_r}_\nu\right)\dfrac{1}{h_{{x_i}_q}} \bigg)\bigg(\prod_{r=1}^{n} h_{{x_r}_\nu}\bigg)\bigg]
 			\end{split}
 			\end{equation} 
for   $j= 2,\cdots, N_1-1 $, \,\, $ k = 2,\cdots, N_2-1 ,\cdots,m = 2,\cdots, N_n-1 $. If one of the indices $ j, k, l,\cdots,m $ is equal to $ 1 $,
\begin{equation}
\begin{split}
&e_{j,k,l\cdots,1,s,\nu,\cdots,m}^{j,k,l,\cdots,1,s,\nu,\cdots,m}\\
&=
\dfrac{1}{l_{j,k,l,\cdots,1,s,\cdots,m}}   {x_i}_{1/2}\dfrac{1}{2}\bigg(\bigg({a_i}_{j,k,l,\cdots,1/2,s\cdots,m}(\tau, \alpha_{j,k,l,\cdots,1,s,\nu,\cdots,m})\\
&+{b_i}_{j,k,l,\cdots,1/2,s\cdots,m}(\tau, \alpha_{j,k,l,\cdots,1,s,\cdots,m})\bigg)\prod_{i\neq r}^{n} h_{{x_r}_\nu} \bigg)+ \\
+&\dfrac{1}{l_{j,k,l,\cdots,1,s,\cdots,m}} \bigg[{x_i}_{1+1/2} {b_i}_{j,k,l,\cdots,1/2,s\cdots,m}(\tau, \alpha_{j,k,l,\cdots,1,s,\cdots,m})\times\\
&\dfrac{{x_i}_{1}^{\beta_{j,k,l,\cdots,q,s\cdots,m}(\tau)} }{{x_i}_{2}^{\beta_{j,k,l,\cdots,1,s\cdots,m}(\tau)}- {x_i}_{1}^{\beta_{j,k,l,\cdots,1,s\cdots,m}(\tau)}}\prod_{i\neq r}^{n} h_{{x_r}_\nu}\\ 
&+\sum_{\underset{i\neq r}{r=1}}^n\bigg( {d_{ir}}_{j,k,l,\cdots,1,s,\nu\cdots,m}(\tau,\alpha_{j,k,l,\cdots,1,s,\nu,\cdots,m})\left({x_i}_1{x_r}_\nu\right) \dfrac{1}{h_{{x_i}_1}} \bigg)h_{{x_i}_1}\prod_{i\neq r}^{n} h_{{x_r}_\nu}\bigg]\\
&- \dfrac{1}{n} c_{j,k,l,\cdots,1,s,\nu,\cdots,m},\\
&\sum_{{j'=j-1}}^{N_1-1} \sum_{{k'=k-1}}^{N_2-1} \sum_{{l'=l-1}}^{N_3-1}\cdots \sum_{{m'=m-1}}^{N_n-1} e_{j',k',l'\cdots,0,s',\nu',\cdots,m'}^{j,k,l,\cdots,1,s,\nu,\cdots,m}\\
& =-\dfrac{1}{l_{j,k,l,\cdots,1,s,\cdots,m}} \bigg[ \sum_{i=1}^n {x_i}_{1/2}\dfrac{1}{2}\bigg(\bigg({a_i}_{j,k,l,\cdots,1/2,s,\nu,\cdots,m}(\tau, \alpha_{j,k,l,\cdots,1,s,\nu,\cdots,m})\\
&-{b_i}_{j,k,l,\cdots,1/2,s,\nu,\cdots,m}(\tau, \alpha_{j,k,l,\cdots,1,s,\nu,\cdots,m})\bigg)\prod_{i\neq r}^{n} h_{{x_r}_\nu} \bigg)\bigg],\\
&\sum_{{j'=j+1}}^{N_1-1} \sum_{{k'=k+1}}^{N_2-1} \sum_{{l'=l+1}}^{N_3-1}\cdots \sum_{{m'=m+1}}^{N_n-1} e_{j',k',l'\cdots,q',s',\nu',\cdots,m'}^{j,k,l,\cdots,1,s,\nu,\cdots,m}
\\
 & = \dfrac{1}{l_{j,k,l,\cdots,1,s,\nu,\cdots,m}} \bigg[ \sum_{i=1}^{n}{x_i}_{1+1/2} {b_i}_{j,k,l,\cdots,1+1/2,s,\nu,\cdots,m}(\tau, \alpha_{j,k,l,\cdots,1,s,\nu,\cdots,m})\times\\
&\bigg(\dfrac{-{x_i}_{2}^{\beta_{j,k,l,\cdots,1,s,\nu,\cdots,m}(\tau)}}{{x_i}_{2}^{\beta_{j,k,l,\cdots,1,s,\nu,\cdots,m}(\tau)}- {x_i}_{1}^{\beta_{j,k,l,\cdots,1,s,\nu,\cdots,m}(\tau)}}\bigg)\prod_{i\neq r}^{n} h_{{x_r}_\nu}\\ 
&-\sum_{\underset{i\neq r}{i,r=1}}^n\bigg( {d_{ir}}_{j,k,l,\cdots,1,s,\nu,\cdots,m}(\tau,\alpha_{j,k,l,\cdots,1,s,\nu,\cdots,m})\left({x_i}_1 {x_r}_\nu\right)\dfrac{1}{h_{{x_i}_1}} \bigg)h_{{x_i}_1}\prod_{i\neq r}^{n} h_{{x_r}_\nu}\bigg].
\end{split}
\end{equation}
The monotonicity of system matrix  $E (\tau,\alpha) $ is given in the following theorem.
  \begin{theorem}\label{tm1}
 	Assume  that the coefficients of $ A $ given by \eqref{matrixA} are positive and $c<0$ \footnote{Indeed  $c$ can be positive but should be less than a certain threshold $c_0>0$}. If $h$ is relatively small then the matrix $E (\tau,\alpha)$  in the system (\ref{stem}) is an $M $-matrix for any $ \alpha_{j,k,l,\cdots,m}  \,\in\,\mathcal{A}^{N}$.
 \end{theorem}
 {\it Proof}  The proof follows the same lines as in Theorem \ref{tm}.}

 \section{Temporal Discretization and optimization problem} \label{sec3}
 This section is devoted to the numerical time discretization method  for the spatially discretized optimization problem after the fitted finite volume method. 
 Let us re-consider the differential equation coupled with optimization problem given in $\left( \ref{pan1}\right) $ by  
 \begin{equation}
 \begin{split}
 \label{nmod}
 \dfrac{d\,{\textbf{v}(\tau)}}{d\,\tau}& = \sup_{\alpha \in \mathcal{A}^N} \left[ A (\tau,\alpha) \textbf{v}(\tau)+ G(\tau,\alpha) \right]\\
 &\textbf{v}(0)\,\,\,\text{given},
 \end{split}
 \end{equation}
 For temporal discretization, we use a constant time step $\Delta t > 0$, of course variable time steps can be used.  
 The temporal grid points given by\, $\Delta t = \tau_{n+1}-\tau_n $~ for ~ $n =  1, 2,\ldots m-1 $. We denote $\textbf{v}(\tau_n) \approx \textbf{v}^n$ , $A^n (\alpha) = A(\tau_n,\alpha)$ and $G^n (\alpha)=G(\tau_n,\alpha).$
 
 For $\theta \,\in \left[\frac{1}{2}, 1\right]$, following \cite{HPFH}, the $\theta$-Method approximation  in time is given by
 \begin{equation}\label{scheme}
 \begin{split}
 & \textbf{v}^{n+1} - \textbf{v}^{n} = \Delta t \,\sup_{\alpha \in \mathcal{A}^N} \left( \theta\,  [A^{n+1} (\alpha)\,\textbf{v}^{n+1} + G^{n+1}(\alpha)] \right. \\ 
 &\left. + (1-\theta)\, [A^{n} (\alpha)\,\textbf{v}^{n} + G^{n}(\alpha)]\right),
 \end{split}
 \end{equation}
 \textcolor{blue}{this also can be written as 
 \begin{equation}\label{schetemp}
 \begin{split}
 \inf_{\alpha \in \mathcal{A}^N} \left( [I + \Delta t \,\theta\, E^{n+1}]\textbf{v}^{n+1} +F^{n+1}(\alpha) + [I + \Delta t \,\theta\, E^{n}]\textbf{v}^{n} +  F^{n}(\alpha)\right) = 0.
 \end{split}
 \end{equation}}
 We can see that to find the unknown $\textbf{v}^{n+1}$, we need also to solve an optimization.  Let 
 \begin{equation}
 \label{opt}
 \alpha^{n+1} \in \left(\underset{\alpha \in \mathcal{A}^N }{arg \sup} \left\lbrace \theta \,\Delta t \left[ A^{n+1}(\alpha)\,\textbf{v}^{n+1} +  G^{n+1}(\alpha )\right] + (1-\theta)\,\Delta t \left[A^{n} (\alpha )\,\textbf{v}^{n} + G^{n}(\alpha)\right] \right\rbrace\right).
 \end{equation}
 Then, the unknown $\textbf{v}^{n+1}$ is solution  of the following equation
 \begin{equation}\label{vu1}
 \begin{split}
 & [ I  - \theta\, \Delta t \,A^{n+1} (\alpha^{n+1})]\,\textbf{v}^{n+1} = [I + (1-\theta)\,\Delta t\, A^{n} (\alpha^{n+1})]\,\textbf{v}^{n} \\
 &+[\theta\, \Delta t \, G^{n+1}(\alpha^{n+1})+(1-\theta) \Delta t\, G^{n}(\alpha^{n+1})] \nonumber,
 \end{split}
 \end{equation} 
 Note that, for $\theta= \dfrac{1}{2}$, we have the \textit{Crank Nickolson scheme} and for $\theta=1$ we have the \textit{Implicit scheme}.
 Unfortunately \eqref{scheme}-\eqref{opt} are nonlinear and coupled and we need to iterate at every time step.
 The following  iterative scheme close to the one in \cite{HPFH} is used.
 \begin{enumerate}
 	\item Let  $ \left( \textbf{v}^{n+1}\right)^0=\textbf{v}^{n}$, 
 	\item Let $ \hat{\textbf{v}}^{k}= \left( \textbf{v}^{n+1}\right)^k$,
 	\item  For $k=0,1,2 \cdots $ until convergence ($\Vert \hat{\textbf{v}}^{k+1}-\hat{\textbf{v}}^{k}\Vert \leq \epsilon$, given tolerance) solve
 	\begin{equation} \label{mi2}
 	\begin{split}
 	&\alpha^{k}_i \in \left(\underset{\alpha \in \mathcal{A}^N }{arg \sup} \left\lbrace \theta \,\Delta t \left[ A^{n+1} (\alpha)\,\hat{\textbf{v}}^k+  G^{n+1}(\alpha) \right]_i  + (1-\theta)\,\Delta t \, \left[A^{n} (\alpha )\,\textbf{v}^{n} + G^{n}(\alpha)\right]_i \right\rbrace\right)\\
 	& \alpha^{k}=(\alpha^{k})_i\\
 	& [ I  - \theta\, \Delta t\,A^{n+1} (\alpha^{k})]\,\hat{\textbf{v}}^{k+1} = [I + (1-\theta)\,\Delta t \, A ^{n}(\alpha^{k})] \textbf{v}^{n} \\
 	&+[\theta\, \Delta t \, G^{n+1}(\alpha^{k})+(1-\theta) \Delta t \, G^{n}(\alpha^{k})],
 	\end{split}
 	\end{equation}
 	\item  Let  $k_l$ being the last iteration in step 3,  set $\textbf{v}^{n+1}:=\hat{\textbf{v}}^{k_l}$,\,\, $\alpha^{n+1}:=\alpha^{k_l}$.	
 \end{enumerate}
 	\textcolor{blue}{The  monotonicity of system matrix of \eqref{schetemp}, more precisely $ [I + \Delta t \,\theta E^{n+1}] $ is given in the following theorem. 	
 	\begin{theorem}\label{uniq}   Under the same assumptions as  in  Theorem \ref{tm},  for any given  $n = 1, 2, \cdots  , m - 1 $,   the
 		system matrix  $ [I + \Delta t \,\theta E^{n+1}] $ in  \eqref{schetemp}  is an $M $--matrix for each $ \alpha \in \mathcal{A}^N. $
 	\end{theorem}
 	\begin{proof}  The proof is obvious.  Indeed  as in Theorem \ref{tm},  $ [I + \Delta t \,\theta E^{n+1}] $ is (strictly) diagonally dominant since $ \Delta t  >  0 $. Then, it is an  $M $--matrix.
 	\end{proof}}
 	\textcolor{blue}{The merit of the
 		proposed method is that it is unconditionally stable in time because of the implicit nature of the time
 		discretization. More precisely, following \cite[ Theorem 6 and Lemma 3]{12}, we can easily prove that the
 		scheme \eqref{scheme} is stable and consistent, so the convergence of the scheme is ensured (see \cite{G1})}	
\section{Application} \label{sec4}
To validate  our method presented in the previous section,  we present  here some numerical experiments. 
All computations were performed in Matlab 2013.

 
 Consider the  following three  dimensional Merton's stochastic control problem  such that  $ \alpha = \alpha_1(t,x) $ is a feedback control in $ [0,1] $ given by 	
 \begin{equation} 
 \label{pb1}
 v(t,x,y,z)=\underset{\alpha\, \in\, [0,1]}{\sup} \mathbb{E}\left\lbrace \dfrac{1}{p}\,x^p(T) \times \dfrac{1}{p}\,y^p(T)\times \dfrac{1}{p}\,z^p(T)\right\rbrace,\,\,\, 0 < p < 1 
 \end{equation}
 s.t.
 \begin{equation}\label{c1}
 \begin{split}
 d x_t &= \left(r_1 + \alpha_t\, (\mu_1-r_1)\right)\,x_t\,dt + \sigma\,x_t \alpha_t\,d\omega_{t},\\
 d y_t &=  \mu_2\,y_t\,dt + \sigma\,y_t\, d\omega_{t},\\ 
 d z_t &= \mu_3\,z_t\,dt + \sigma\,z_t\, d\omega_{t}.
 \end{split}
 \end{equation}
   $r_1$, $\mu_1$, $\mu_2$, $\mu_2$, $\sigma $ are positive constants, $ x_t,\, y_t,\, z_t\,\in\,\mathbb{R} $. We assume that $ \mu_1 > r_1 $. For the problem \eqref{pb1}-\eqref{c1}, the corresponding HJB equation is given by
 \begin{equation}\label{65}
 \begin{cases}
 \dfrac{d\,v(t, x, y, z)}{d\,t} + \underset{\alpha \in [0, 1]}{\sup} \left[ L^\alpha \,v(t, x, y, z)\right] = 0 \quad\text{on} \  [0,T)\times \mathbb{R}\times \mathbb{R}\times \mathbb{R}\\
 v(T, x, y, z ) = \dfrac{x^p}{p}\times \dfrac{y^p}{p}\times \dfrac{y^p}{p}, \,\, \,x,\, \,y,\,\,\, z\,\in \mathbb{R}_+
 \end{cases}
 \end{equation}	  
 where  \begin{equation*} 
 \begin{split}
 &L^\alpha \,v(t, x, y, z) \\
 &=\dfrac{1}{2}\,\sigma^2\,\alpha^2\,x^2 \dfrac{d^2 v(t, x, y, z)}{d x^2} + \dfrac{1}{2}\,\sigma^2\,y^2 \dfrac{d^2 v(t, x, y, z)}{d y^2} +\dfrac{1}{2}\,\sigma^2\,z^2 \dfrac{d^2 v(t, x, y, z)}{d z^2} \\
 & + \sigma^2\,\alpha\,x\,y\, \dfrac{d^2 v(t, x, y, z)}{d x\partial y} + \sigma^2\,\alpha\,x\,z\, \dfrac{d^2 v(t, x, y, z)}{d x\,d z} + \sigma^2\,z\,y\, \dfrac{d^2 v(t, x, y, z)}{d z\,d y} \\
 & + (r_1 + (\mu_1 -r_1)\alpha )\,x\, \dfrac{d v(t, x,y,z)}{d x} + \mu_2\,y\, \dfrac{d v(t, x, y, z)}{d y} + \mu_3\,z\, \dfrac{d v(t, x, y, z)}{d z}.
 \end{split}
 \end{equation*}
 \begin{equation} 
 \dfrac{d v(t,x,y,z) }{\partial t} + \sup_{\alpha \, \in\, [0,1]}\left[ \nabla\cdot \left( k(t,x,y,z,\alpha) (v(t,x,y,z))\right) + c(t,x,y,z,\alpha)\,v(t,x,y,z) \right] = 0,
 \end{equation}	  
 and the different variable in (\ref{65}) is given by  $$k(v(t,x,y,z)) = A(t,x,y,z,\alpha)\nabla v(t, x, y, z)+ b(t,x,y,z,\alpha)\,v(t, x, y, z)$$\, is the flux,\, $b = (x\,b_1, y\,b_2, z\,b_3)^T$,
 \[
 A=\left[ \begin{array}{ccc}
 a_{11} & a_{12} & a_{13} \\
 a_{21} & a_{22}& a_{23}  \\
 a_{31} & a_{32}& a_{33}
 \end{array} \right].
 \]
 with
 \begin{equation}
 \begin{split}
 a_{11} &= \dfrac{1}{2}\sigma^2\,\alpha^2\,x^2 ,~ a_{22} = \dfrac{1}{2}\sigma^2\,y^2, ~ a_{33} = \dfrac{1}{2}\sigma^2\,z^2, \\
 a_{12} & = a_{21} = \dfrac{1}{2}\sigma^2\,\alpha\,x\,y, ~ a_{13} = a_{31} = \dfrac{1}{2}\sigma^2\,\alpha\,x\,z,\\
 a_{23} & = a_{32} = \dfrac{1}{2}\sigma^2\,y\,z.
 \end{split}
 \end{equation}
 	\begin{equation}
 	\begin{split}
 	\begin{cases}
 	& b_1(t,x,y,z,\alpha)=  r_1 + (\mu_1-r_1)\,\alpha - \sigma^2\,\alpha - \sigma^2\,\alpha^2\\
 	& b_2(t,x,y,z,\alpha) =  \mu_2 -\dfrac{1}{2}\sigma^2\alpha - \dfrac{3}{2}\sigma^2 \\
 	& b_3(t,x,y,z,\alpha) = \mu_3 -\dfrac{1}{2}\sigma^2\,\alpha - \dfrac{3}{2}\sigma^2\\
 	& c(t,x,y,z,\alpha) = - \left[ r_1 + (\mu_1-r_1)\,\alpha - 2\,\sigma^2\,\alpha- \sigma^2\,\alpha^2
 	+ \mu_2\,+\mu_3 - 3\,\sigma^2 \right].
 	\end{cases}
 	\end{split}
 	\end{equation}
 The domain where we compare the solution is $ \Omega =\left[0, x_{{max}}\right] \times \left[0, y_{{max}}\right]\times \left[ 0, z_{{max}}\right]$. For each simulation,
 the exact or reference solution is the analytical solution using Ansatz method as we are going to develop in the next section.
 \subsection{Analytical solution using Ansatz method} \label{sec5}
 Here we  propose the analytical solution using the Ansatz decomposition. Let set the Ansatz decomposition of $v$
 \begin{align}
 v(t, x, y, z)=\psi(t)\times u(x)\times u(y)\times u(z),
 \end{align}
 where \,\,$ u(x) = \dfrac{x^p}{p},$  $\,\, 0 < p < 1$,  $\forall\, x \,\in \mathbb{R}_+ $ is the power utility function. The different derivative of  $v(t,x,y,z) $
 gives us
 \begin{equation}
 \begin{split}
 \begin{cases}
 \dfrac{d v(t,x,y,z)}{d x}=\psi(t) \dfrac{d u(x)}{d x}\, u(y)\,u(z);\,\,
 \dfrac{d v(t,x,y,z)}{d y}=\psi(t) \dfrac{d u(y)}{d y}\, u(x)\,u(z)\\
 \dfrac{d v(t,x,y,z)}{d z}=\psi(t) \dfrac{d u(z)}{d z}\, u(x)\,u(y);\,\,
 \dfrac{d v(t,x,y,z)}{d t}=\psi'(t)\, u(x)\,u(y)\,u(z)\\
 \dfrac{d^2 v(t,x,y,z)}{d x^2} =\psi(t) \dfrac{d^2 u(x)}{d x^2}\, u(y)\,u(z);\,\,
 \dfrac{d^2 v(t,x,y,z)}{d y^2} =\psi(t) \dfrac{d^2 u(y)}{d y^2}\, u(x)\,u(z)\,\\
 \dfrac{d^2 v(t,x,y,z)}{d z^2} =\psi(t) \dfrac{d^2 u(z)}{d z^2}\, u(x)\,u(y);\,\,
 \dfrac{d^2 v(t,x,y,z)}{d x\,d y} =\psi(t) \dfrac{d u(x)}{d x}\,\dfrac{d u(y)}{d y}\, u(z)\\
 \dfrac{d^2 v(t,x,y,z)}{d x\,d z} =\psi(t) \dfrac{d u(x)}{d x}\,\dfrac{d u(z)}{d z}\, u(y);\,\,
 \dfrac{d^2 v(t,x,y,z)}{d y\,d z} =\psi(t) \dfrac{d u(y)}{d y}\,\dfrac{d u(z)}{d z}\, u(x),
 \end{cases}
 \end{split}
 \end{equation}
 plugging into (\ref{65}), we get
 \begin{equation}
 \begin{split}
 \begin{cases}
 &\dfrac{d\,\psi(t)}{d t} \, u(x)\,u(y)\,u(z) \\
 & + r_1\,x \,\psi(t)  \dfrac{d \left(u(x)\,u(y)\,u(z)\right)}{d x} 
 +\underset{\alpha \in A}{\sup}\left[  \alpha (\mu_1-r_1)\,x\,\psi(t) \dfrac{d \left(u(x)\,u(y)\,u(z)\right) }{d x} \right.\\
 &\left.+\mu_2\,y\,\psi(t) \dfrac{d \left(u(x)\,u(y)\,u(z)\right) }{d y}+\mu_3\,z\,\psi(t) \dfrac{d \left(u(x)\,u(y)\,u(z)\right) }{d z} \right.\\
 &\left. + \dfrac{1}{2} \sigma^2\,\alpha^2\,x^2\psi(t) \dfrac{d^2 \left(u(x)\,u(y)\,u(z)\right)}{d x^2} 
 + \dfrac{1}{2} \sigma^2\,y^2 \,\psi(t) \dfrac{d^2 \left(u(x)\,u(y)\,u(z)\right)}{d y^2} \right.\\
 &\left. + \dfrac{1}{2} \sigma^2\,z^2\psi(t) \dfrac{d^2 \left(u(x)\,u(y)\,u(z)\right)}{d z^2}  +  \sigma^2\,\alpha\,\psi(t)\,x\,y\, \dfrac{d^2 }{d x\,d y} \left(u(x)\,u(y)\,u(z)\right) \right. \\
 & \left.+  \sigma^2\,\alpha\,\psi(t)  \,x\,z\, \dfrac{d^2 }{d x\,d z} \left(u(x)\,u(y)\,u(z)\right)+ \sigma^2\,\psi(t)  \,y\,z\, \dfrac{d^2 }{d y\,d z} \left(u(x)\,u(y)\,u(z)\right)\right]=0 \\
 &\psi(T)=1,\;  (\text{since}~~ v(T,x,y,z)=\psi(T)\,u(x)\,u(y)\,u(z)= u(x)\,u(y)\,u(z))
 \end{cases}
 \end{split}
 \end{equation}
  We then  obtained
 \begin{equation}
 \begin{split}
 & \dfrac{d \psi(t)}{d t}+ p\,\rho\,\psi(t) = 0 ~~\text{where}~~\\
 & \rho = \underset{\alpha \in \,[0,1]}{\sup}\left[ r_1  + (\mu_1-r_1)\,\alpha + \mu_2+\mu_3\,+ \dfrac{1}{2} \sigma^2\,\alpha^2\,(p-1)\right.\\  
 & \left. + \sigma^2\,(p-1)
 + 2\,\sigma^2\,\alpha\,p  +\sigma^2\,p\right],\\
 & \psi(T) = 1.
 \end{split}
 \end{equation}
 So by setting $ \tau = T-t $, the analytical value function for Ansatz method is then equal to 
 \begin{equation}\label{pl}
 v\left(\tau^n, x_i, y_j, z_k\right) =  e^{p \times (n \times \Delta t - T) \times \rho} \times \dfrac{(x_i)^p}{p} \times \dfrac{(y_j)^p}{p}\times \dfrac{(z_k)^p}{p}, \,\,\text{with}\,\,0 < p < 1.
 \end{equation}
We use the  following $L^2\left( [0,T]\times \Omega\right)$ norm of the absolute error 
 \begin{equation}
 \begin{split}
 &\ \left\| v^m-v\right\|_{L^2\left(  [0,T]\times \Omega\right)}\\
 & = \left(\sum_{n = 0}^{m-1} \sum_{i = 1}^{N_1-1} \sum_{j = 1}^{N_2-1} \sum_{k = 1}^{N_3-1} (\tau_{n+1}-\tau_n)\times l_{i,j,k}\times (v_{i,j,k}^n - v\left( \tau^n, x_{i}, y_j, z_k\right))^2 \right)^{1/2},
 \end{split}
 \end{equation}
 where $v^m$ is the numerical approximation of $v$ computed from  our numerical scheme. 
For our computation, we us we have  $\Omega=[0,1/2]\times [0,1/4] \times [0,1/2]$  for computational domain  with $N_1= 10$, $N_2 = 10$, $N_3= 10$, $r_1 = 0.0449$, $\mu_1 = 0.0657$, $\mu_2 = 0.067$, $\mu_3 = 0.066$, $\sigma = 0.2537$, $p= 0.13$ and   $T=1 $. Figure \ref{5.5d} shows the structure of  the matrix $ A $ after space discretisation with the fitted volume method.  As you can observe the structure of the matrix is similar to the one from finite difference method.
  \textcolor{blue}{Figure \ref{control} shows the optimal investment policy as function of $ x $ while using the fitted scheme.   The optimal investment policy  for finite difference method  is quite similar. 
  Indeed  the optimal parameter $\alpha$ is independent of $y$ and  $z$.
 The controller is
  the solution of (\ref{pb1}). 
It is computed with the numerical procedure as outlined in Section \ref{sec3}. We have also  found that in overall the value the  maximum  number of iterations  in our optimisation algorithm is 3  in both  fitted scheme and finite difference scheme.}
  \begin{figure}[!h]
  	   	\begin{minipage}[b]{0.5\linewidth}
  	\centering\includegraphics[scale=0.42]{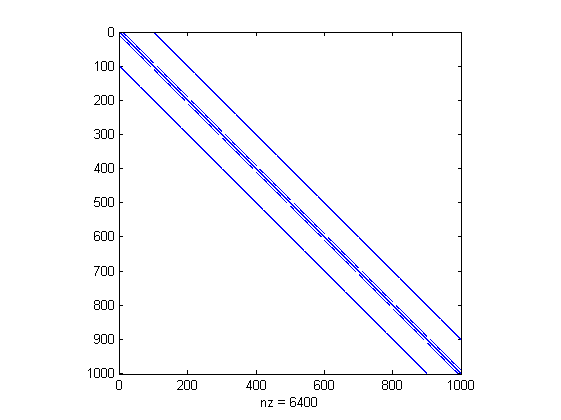}
  	\caption{Structure of the matrix $ A $
  		at time $T=1$.}
  	\label{5.5d}
  	    	\end{minipage}\hfill
  	    	\begin{minipage}[b]{0.5\linewidth}
  	    		\centering\includegraphics[scale=0.4]{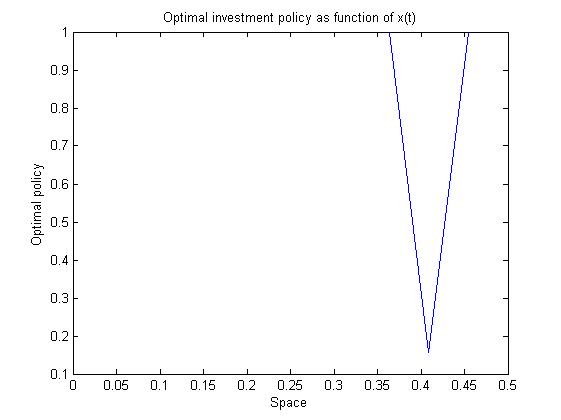}
  	    		\caption{Optimal investment policy}{ at time $T=1$. }
  	    		\label{control}
  	    	\end{minipage}
  \end{figure}
 
 \newpage
 We compare the fitted finite volume and the finite difference method in Table \ref{2i}
 \begin{table}[ht]
 	\begin{center}
 			\begin{tabular}{|c|c|c|c|c|}
 			\hline
 			Time subdivision &  $200 $ &  $150 $ & $ 100 $ &  $50 $ \\ 
 			\hline Error of  fitted finite volume  method & 4.65 E-01  & 6.31 E-01  & 8.63 E-01 & 1.30 E-00\\
 			\hline Error of  finite difference method & 5.15 E-01  & 6.98 E-01  & 9.21 E-01 & 1.36 E-00 \\  
 			\hline
 		\end{tabular}
 		\vspace*{0.1cm}
 		\caption{Comparison of the  implicit fitted finite volume method and implicit finite difference method. For the parameters, we have used $N_1= 10$, $N_2 = 10$, $N_3= 10$, $r_1 = 0.0449$, $\mu_1 = 0.0657$, $\mu_2 = 0.067$, $\mu_3 = 0.066$,
 			$\sigma = 0.2537$, $p= 0.13$ and   $T=1 $.}
 		\label{2i}
 	\end{center}
 \end{table}
 
\textcolor{blue}{ Figure \ref{5.5dd} shows the structure of  the matrix $ A $  and Figure \ref{controll} shows the optimal investment policy as function of $ x $. The controller is
 	the solution of (\ref{pb1}). 
 	It is computed with the numerical procedure as outlined in Section \ref{sec3}.
 \begin{figure}[!h]
 	\begin{minipage}[b]{0.5\linewidth}
 		\centering\includegraphics[scale=0.42]{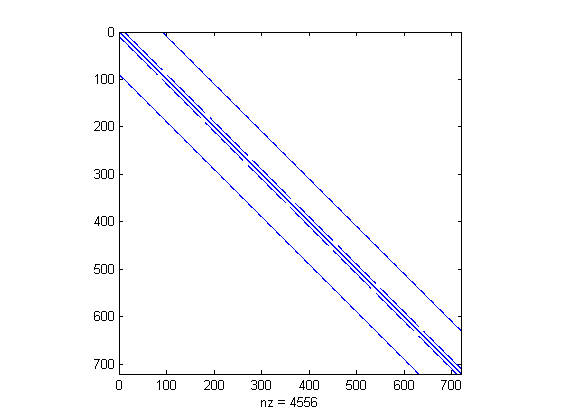}
 		\caption{Structure of the matrix $ A $
 			at time $T=1.5$.}
 		\label{5.5dd}
 	\end{minipage}\hfill
 	\begin{minipage}[b]{0.5\linewidth}
 		\centering\includegraphics[scale=0.4]{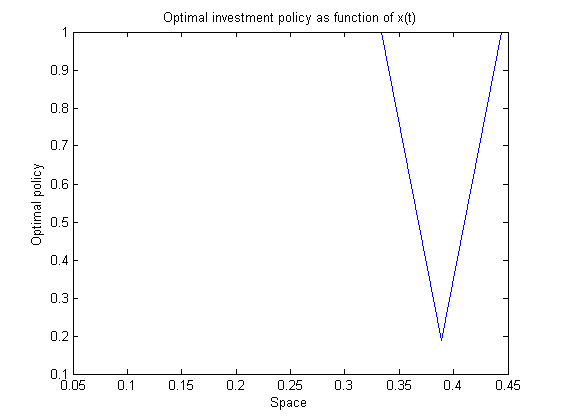}
 		\caption{Optimal investment policy}{ at time $T=1.5$. }
 		\label{controll}
 	\end{minipage}
 \end{figure} }
 

 In  Table \ref{2ii}, we  have used   $\Omega=[0,1/2]\times [0,1/4] \times [0,1/2]$  for computational domain with the following parameters $N_1= 8$, $N_2 = 9$, $N_1= 10$, $r_1 = 0.0449$, $r_2 = 0.0448/3$, $r_3 = 0.0447$, $\mu_1 = 0.0657$, $\mu_2 = 0.0656$, $\mu_3 = 0.0655$, $\sigma = 0.2537$, $p= 0.17$ and   $T=1.5 $.
 \begin{table}[!h]
 	\begin{center}
 		
 		\begin{tabular}{|c|c|c|c|c|}
 			\hline
 			Time subdivision &  $200 $ &  $150 $ & $ 100 $ &  $50 $ \\ 
 			\hline Error of   the fitted  finite volume method  & 2.24 E-01  & 2.30 E-01  & 3.99 E-01 & 5.97 E-01\\ 
 			\hline Error of  the finite difference method & 2.40 E-01  & 3.18 E-01  & 4.12 E-01 & 5.99 E-01 \\ 
 			\hline
 		\end{tabular}
 		\vspace*{0.1cm}
 		\caption{Comparison of the  implicit fitted finite volume method  and implicit finite difference method. For the parameters, we have used  $N_1= 8$, $N_2 = 9$, $N_1= 10$, $r_1 = 0.0449$, $r_2 = 0.0448/3$, 
		$r_3 = 0.0447$, $\mu_1 = 0.0657$, $\mu_2 = 0.0656$, $\mu_3 = 0.0655$, $\sigma = 0.2537$, $p= 0.17$  and   $T=1.5 $.}
 		\label{2ii}
 	\end{center}
 \end{table} 
 \newpage
 Table \ref{2i} and Table \ref{2ii} display the numerical errors of finite volume method  and finite difference method. \textcolor{blue}{ By fitting the data from Table \ref{2i} and Table \ref{2ii}, we found that the convergence order in time  $1$ for the fitted finite volume method and the finite difference method}. From the two tables, we can observe a slight accuracy of the  implicit fitted finite volume comparing to the implicit finite difference method,  thanks to the fitted technique.

 \section{Conclusion}
 \label{sec6}
 We have introduced a novel scheme  based on  finite volume method with fitted  technique  to solve \textcolor{blue}{high} dimensional   stochastic optimal  control problems ($n\geq 3$).
 The optimization  problem  is solved  at every time step using iterative method. We  have shown that the system matrix of the resulting non linear system is
 an $M $-matrix  and  therefore the maximum principle is preserved  for  the discrete system obtained after the fitted finite volume spatial discretization. 
 Numerical experiments  are used to demonstrate the  accuracy of the novel  scheme comparing to the standard finite difference method.
\section{Conflicts of interest/Competing interests }
We have no conflicts of interest to declare.

\end{document}